\documentclass[11pt]{amsart}
\usepackage[applemac]{inputenc}

\usepackage{fullpage}
\usepackage{
 ulem,
 amsfonts}
 \usepackage{amsmath}
\usepackage{amsmath,esint, amssymb, amsthm,mathbbol,upgreek,exscale}

 \usepackage{mathtools}

\usepackage[colorlinks=true, pdfstartview=FitV, linkcolor=blue, citecolor=blue,
 urlcolor=blue]{hyperref}
\usepackage{color}

\date{}

\newcommand{\CC}{\mathbb{C}}  
  
\newcommand{\RR}{\mathbb{R}}

\theoremstyle{plain}

\numberwithin{equation}{section}
\newtheorem{theorem}{Theorem}[section]
\newtheorem{proposition}[theorem]{Proposition}

\newtheorem{remark}[theorem]{Remark}

\usepackage[textmath,displaymath,floats,graphics,delayed]{preview}
\title{Existence of corotating  asymmetric vortex pairs for Euler equations}
\author[Z. Hassainia]{Zineb Hassainia}
\address{NYU Abu Dhabi\\
Saadiyat Marina District - Abu Dhabi, United Arab Emirates.
}
\email{zh14@nyu.edu}

\author[T. Hmidi]{Taoufik Hmidi}
\address{IRMAR, Universit\'e de Rennes 1\\ Campus de
Beaulieu\\ 35~042 Rennes cedex\\ France}
\email{thmidi@univ-rennes1.fr}

\begin{document}
 \begin{abstract}
In this paper, we study the existence of co-rotating and counter-rotating unequal-sized  pairs of simply connected patches for Euler equations.  In particular, we prove the existence of  curves of steadily co-rotating and counter-rotating asymmetric vortex pairs passing through a point vortex pairs with unequal circulations. 

\end{abstract}
 \maketitle

\tableofcontents

\section{Introduction}
In this paper we shall be concerned with the dynamics of asymmetric  vortex pairs for  the two-dimensional incompressible Euler equations. These equations describe the motion of an ideal fluid and take the form
\begin{equation} \label{eqn:omega}
\left\{ \begin{array}{ll}
  \partial_t \omega + v \cdot \nabla \omega = 0\quad (t,x)\in \mathbb{R}_+\times \mathbb{R}^2,&\\ 
   v= -\nabla^\perp(-\Delta)^{-1}\omega,&\\
   \omega_{|t=0}=\omega_0,
  \end{array}\right.
\end{equation}
where $v = (v_1, v_2)$ refers to the velocity field and its   vorticity  $\omega$ is given  by the
scalar $\omega = \partial_1v_2 -\partial_2v_1.$
The second equation in \eqref{eqn:omega} is nothing but the Biot-Savart law which can be written in the complex form:
\begin{align*}
v(t,z) &=\frac{i}{2\pi }\int_{ \mathbb{R}^2}\frac{\omega(t,\zeta)}{\overline{\zeta}-\overline{z}}dA(\zeta), \quad\forall z\in \CC,
\end{align*}
where we identify  $v = (v_1, v_2)$ with the complex valued-function $v_1 + iv_2$ and  $dA$ denotes  the planar Lebesgue measure.
 The global existence and uniqueness of solutions with initial integrable and bounded vorticity was established a long time ago by Yudovich \cite{Y}. He  proved that the system \eqref{eqn:omega} admits a unique global solution in the weak sense, provided that the initial vorticity $\omega_0$ lies in $L^1\cap L^\infty$. In this setting,  we can rigorously deal  with the  so-called  vortex patches, which are the characteristic function of bounded domains.  This specific structure is preserved along the time and the solution $\omega(t)$ is uniformly distributed in the bounded domain $D_t$, which is  the image by the flow mapping of the initial domain $D_0$.  
When  $D_0$ is a disc then the  vorticity is stationary. However, elliptical vortex patches undergo a perpetual rotation about their centers without changing the shape. This discovery goes back to Kirchhoff \cite{K}  and till now, the ellipses are  the only explicit examples with such properties in the setting of vortex patches.  The existence  of general class of  implicit rotating patches, called also V-states, was discovered numerically by Deem and Zabusky
\cite{DZ}.  Few years later,  Burbea \cite{Bu} gave  an analytical proof  using complex analytical tools and bifurcation theory to show the existence of  a countable family of V-states  with $m$-fold symmetry for each integer $m \geq 2$. More precisely, 
the rotating patches appear as a collection of one dimensional branches bifurcating from Rankine vortices  at the  discrete  angular velocities set   $\big\{\frac{ m-1}{2m}$, $m \geq  2\big\}$.
These local branches were extended very recently  to global ones in \cite{HMH}, where the minimum value on the patch boundary of the angular fluid velocity becomes arbitrarily small near the end of each branch. The regularity of the V-states boundary  has been conducted  in a series  of papers   \cite{ CCG, CCG1,HMH, HMV}. The existence of small loops in the bifurcation diagram has been proved recently in \cite{HR}.
From numerical point of view,  Wu, Overman, and Zabusky \cite{WOZ} went further along the same branches  and found singular limiting solutions with $90^\circ$ corner. We also refer to the paper \cite{O} where it is proved that corners with right angles is the only plausible scenario for the limiting V-states. 

It is worth pointing out that Burbua's approach  has been intensively exploited in the few last years in different directions. For instance, this was implemented  to prove the existence of rotating patches close to Kirchhoff's ellipses \cite{CCG4,HM2}, multi-connected patches \cite{ HHFMV, HFMV}, patches in bounded domains \cite{HHFMV}, non trivial rotating smooth solutions \cite{CCG2, CCG3} and rotating vortices with non uniform densities \cite{GHS} .
We mention that many of these results apply not only to the Euler equations but also to more singular nonlinear transport equations as the inviscid surface quasi-geostrophic equations or the quasi-geostrophic shallow-water equations, but with much more involved computations; in this context also see \cite{CCG4,CCG3,CCG,CCG1,HHF,DHR,G-S-Y,HH,HM3,HMV2}.
 
It is important to emphasize that all of the aforementioned  analytical results treat connected patches.
However, for the  disconnected ones the bifurcation arguments discussed above are out of use.  The main objective of this paper is to deal with vortex pairs moving without changing the shape.
One of the very simplest nontrivial vortex equilibria is given by a pair of point vortices with  magnitude $\gamma_1$ and $\gamma_2$ and  far away at a distance $d$. It is well known that if the circulations are with opposite signs $\gamma_1=-\gamma_2$  then the system travels  in a rectilinear motion with a constant speed $U_0=\displaystyle \frac{\gamma_1}{2\pi d}$, otherwise  the pair of point vortices   rotates steadily with the angular \mbox{velocity $\Omega_0=\displaystyle\frac{\gamma_1+\gamma_2}{2\pi d^2}$.}

Coming back to the emergence of steady  disconnected  vortex patches, translating vortex pairs of symmetric patches were discovered numerically by Deem and Zabusky \cite{DZ} and Pierrehumbert \cite{P} where they conjectured the existence of a curve of translating symmetric pair of simply connected patches emerging from two point vortices. A similar study was established by  Saffman and Szeto in \cite{SS} for the co-rotating vortex pairs, where two symmetric patches with the same circulations rotate about the centroid of the system with constant angular velocity.  In the same direction Dritschel \cite{D} calculated numerically the  V-states of asymmetric  vortex pairs  and discussed their linear stability.

The analytical  study of the corotating vortex pairs of patches was conducted  by Turkington \cite{T} using variational principle. However, this approach does not give sufficient information on the topological structure of each vortex patch and the uniqueness problem is left open. In the same direction  Keady \cite{K} implemented  the same approach   to prove the existence part of  translating vortex pairs of symmetric patches. Very recently, Hmidi and Mateu \cite{HM} gave direct proof confirming the numerical experiments  and showing the existence of co-rotating and counter-rotating vortex pairs using the contour dynamics equations combined with a desingularization of the point vortex pairs and the application of the implicit function theorem.

The main concern of this work is to explore the existence of  of co-rotating and counter-rotating asymmetric  vortex pairs using the contour dynamics equations. Before stating our result we need to make some notation. 
Let $ \varepsilon \in (0,1), \gamma_1, \gamma_2\in \mathbb{R}, b_1,b_2\in \mathbb{R}_+$ and $d>2(b_1+b_2)$. Consider two small simply connected domains $D_1^\varepsilon$  and $D_2^\varepsilon$ containing the origin and contained in the open ball $B(0, 2)$ centered at the origin and with radius 2. Define
$$
\omega_{0,\varepsilon}=\frac{\gamma_1}{\varepsilon^2b_1^2}\chi_{\tilde{D}_1^\varepsilon}+\frac{\gamma_2}{\varepsilon^2b_2^2}\chi_{\tilde{D}_2^{\varepsilon}},
$$
with
$$
\tilde{D}_1^\varepsilon=\varepsilon b_1 D_1^\varepsilon\quad\textnormal{and}\quad  \tilde{D}_2^{\varepsilon}=-\varepsilon b_2  D_2^\varepsilon+d.
$$
Informally stated, our main existence result is the following
\begin{theorem}\label{thm:informal} 
There exists $\varepsilon_0 > 0$ such that the following results hold true.
\begin{enumerate} 
\item For any $\gamma_1,\gamma_2\in\mathbb{R}$ such that $\gamma_1+\gamma_2\neq0$ and any $\varepsilon\in (0, \varepsilon_0]$ there exists  two strictly convex domains $D^\varepsilon_1$ and $D^\varepsilon_2$ at least of class $C^1$ such that $\omega_{0,\varepsilon}$  generates a co-rotating vortex pair for \eqref{eqn:omega}.
\item For $\gamma_1\in\mathbb{R}$ and any $\varepsilon\in (0, \varepsilon_0]$ there exists  $\gamma_2=\gamma_2(\varepsilon)$ and   two strictly convex domains $D^\varepsilon_1$, $D^\varepsilon_2$ at least of class $C^1$  such that $\omega_{0,\varepsilon}$  generates a counter-rotating vortex pair for \eqref{eqn:omega}. 
\end{enumerate}
\end{theorem}

\vspace{0,2cm}
\begin{remark}
 The domains $D^\varepsilon_j$, $j=1,2$ are  small perturbations of the unit  disc. Moreover, as a by-product of the proofs the corresponding conformal parametrization ${\phi}_j^\varepsilon:\mathbb{T}\to \partial  {D}^\varepsilon_j$  belongs to $C^{1+\beta}$ for any $\beta\in (0,1) $, and has the Fourier asymptotic  expansion 
 \begin{align*}
\phi_j(\varepsilon,w)&=w+\delta_j\Big(\frac{\varepsilon b_j }{ d}\Big)^2\overline{w}+\frac{\delta_j}{2}\Big(\frac{\varepsilon b_j }{ d}\Big)^3\overline{w}^2+\frac{\delta_j}{3}\Big(\frac{\varepsilon b_j}{d}\Big)^4\Big(\overline{w}^3+6\big(1+\delta_j\big)\overline{w}\Big)\\ &+\frac{\delta_j}{4}\Big(\frac{\varepsilon b_j}{d}\Big)^5\Big(\overline{w}^4+3\big(1+\delta_j\big)\overline{w}^2\Big)+o(\varepsilon^5) ,\quad j=1,2,
\end{align*}
where $\delta_j=\frac{ \gamma_{3-j}}{ \gamma_j}$ in the co-rotating case and $\delta_j=-1$ in the translating one. In addition, the angular velocity has the expansion
$$
\Omega(\varepsilon)=\frac{\gamma_1+\gamma_2}{2d^2}+\frac{\varepsilon^4}{2d^6}\Big(
{\gamma_1} b_2^4+{\gamma_2} b_1^4 \Big)+o(\varepsilon^4),
$$
and the center of rotation has the expansion
$$
Z(\varepsilon)=\frac{\gamma_2 d}{\gamma_1+\gamma_2}+\frac{\varepsilon^4}{d^3(\gamma_1+\gamma_2)^2}\Big(
\frac{\gamma^3_2}{\gamma_1} b_1^4-\frac{\gamma_1^3}{\gamma_2} b_2^4 \Big)+o(\varepsilon^4).
$$
In the case of  translation, the speed has the expansion
$$
U(\varepsilon)=\frac{\gamma_1}{2d}\Big(1+\frac{\varepsilon^4}{d^4}(2b_1^4+b_2^4)\Big)+o(\varepsilon^4)
$$
and  the vorticity $\gamma_2$  has the expansion 
$$
\gamma_2(\varepsilon)={\gamma_1}\Big(1+\frac{\varepsilon^4}{d^4}(b_1^4-b_2^4)\Big)+o(\varepsilon^4).
$$

\end{remark}
\begin{remark}

As we shall see later, the case  $\varepsilon=0$ corresponds to the vortex point system. This allows to recover the classical result stating  that two point vortices at distance $2d$ and  magnitudes ${2\pi}\gamma_1$ and ${2\pi}\gamma_2$
rotate uniformly about their centroid with the angular velocity $\Omega_0=\frac{\gamma_1+\gamma_2}{2d^2}$ provided that $\gamma_1+\gamma_2\neq0$ . However, they translate uniformly with the speed $U=\frac{\gamma_1}{2d}$ when $\gamma_1$ and $\gamma_2$ are opposite.
\end{remark}
\begin{remark}
The interval $(0, \varepsilon_0]$ is uniform in $b_1$ and $b_2$ and therefore we my recover the point vortex-vortex patch configuration by letting $b_1$ and  $b_2$ go to $0$. 
\end{remark}
The proof of this theorem is done in the spirit of the work \cite{HM} using contour dynamics reformulation. It is a known fact that the initial vorticity $\omega_{0,\varepsilon}$ with the velocity $v_{0,\varepsilon}$ generates a rotating solution, with constant angular velocity $\Omega$ around the point $Z$ on the real axis, if and only if
\begin{equation}\label{eq:rot}
\textnormal{Re}\bigg\{\Big(-i\Omega\big(\overline{z}-Z\big)-\overline{v(z)}\Big)\vec{n}\bigg\}=0\quad \forall z\in \partial \tilde{D}_1^\varepsilon \cup \partial \tilde{D}_2^{\varepsilon},
\end{equation}
where $\vec{n}$ is the exterior unit normal vector to the boundary at the point $z$. According to the  Biot-Savart law combined with Green-Stokes formula , we can write
\begin{align*}
\overline{v(z)} &=\frac{i\gamma_1}{2\varepsilon^2b_1^2}\fint_{\partial \tilde{D}_1^\varepsilon}\frac{\overline{\xi}-\overline{z}}{\xi-z}d\xi+\frac{i\gamma_2}{2\varepsilon^2 b_2^2}\fint_{\partial \tilde{D}_2^\varepsilon}\frac{\overline{\xi}-\overline{z}}{\xi-z}d\xi, \quad\forall z\in \CC. 
\end{align*}
Following the same line of \cite{HM} we can remove the singularity in $\varepsilon$ by  slightly perturbing  the unit disc with a small amplitude of order $\varepsilon$ and taking advantage  of its symmetry. Indeed, we seek for conformal parametrization of the boundaries  $\phi_j:\mathbb{D}^c\to [D_j^\varepsilon ]^c$, in the form
$$
\phi_j(w)= w+\varepsilon b_j f_j(w), \quad \textnormal{with}\quad f_j(w)=\sum_{n\geq 1} \frac{a_n^j}{w^n}, \quad a_n^j\in \RR, j=1,2.
$$
Straightforward computations and substitutions we find that  the conformal mappings are subject to two coupled nonlinear equations defined as follows: for all $w\in \mathbb{T}$ and $ j\in\{1,2\}$,
\begin{equation}\label{eq:steady}
F_j\big(\varepsilon,\Omega,Z,f_1,f_2\big)=\textnormal{Im}\bigg\{\bigg(2\Omega\Big(\varepsilon b_j\overline{\phi_j(w)}+(-1)^jZ-(j-1)d\Big)+J_j^\varepsilon(w)\bigg)w\phi'_j(w)\bigg\}=0,
\end{equation}
  where $J_\varepsilon$ ranges in  the H\"older space $C^\alpha$ and can be extended for $\varepsilon\in(-\varepsilon_0,\varepsilon_0)$ with $\varepsilon_0>0$. In addition, the functional $F=(F_1,F_2):(-\frac12,\frac12)\times\mathbb{R}\times\mathbb{R}\times X\to Y$ is well-defined and it is of class $C^1$ where
 $$
X\triangleq\bigg\{f\in \big( C^{1+\alpha}(\mathbb{T})\big)^2,\; f(w)=\sum_{n\geq 1}A_n\overline{w}^n,\; A_n\in \RR^2, w\in \mathbb{T}\bigg\}
$$
and 
$$
Y=\bigg\{g\in \big( C^{\alpha}(\mathbb{T})\big)^2,\; g=\sum_{n\geq 1}C_n e_n,\; C_n\in \RR^2, w\in \mathbb{T}\bigg\},\quad e_n(w)=\textnormal{Im}\big\{{w}^n\big\}.
$$
In order to apply the Implicit Function Theorem  we compute the linearized operator around the point vortex pairs leading to 
\begin{align*}
D_{((f_1,f_2)}F(0,\Omega,Z,0,0)(h_1,h_2)(w)&=-\begin{pmatrix}
{\gamma_1}\textnormal{Im}\Big\{h'_1(w)\Big\} \\
{\gamma_2}\,\textnormal{Im}\Big\{h'_2(w)\Big\}
\end{pmatrix},
\end{align*}
This operator is not invertible from $X$ to $Y$ but it does from $X$ to $\tilde{Y}$, where
\begin{equation*}
\tilde{Y}\triangleq \bigg\{g\in Y\;:\; C_1=\begin{pmatrix}
0 \\
0 
\end{pmatrix}\bigg\}.
\end{equation*}
The idea to remedy to this defect is to consider the operator: for any $h=(\alpha_1,\alpha_2,h_1,h_2)\in \mathbb{R}\times\mathbb{R}\times X$

\begin{align*}
D_{(\Omega,Z,f_1,f_2)}F (0,\Omega_0,Z_0,0,0)h(w)=&-\dfrac{2\alpha_1 d}{\gamma_1+\gamma_2} \begin{pmatrix}
\gamma_2  \\
\gamma_1
\end{pmatrix}\textnormal{Im}\big\{{w}\big\}- \dfrac{\alpha_2 (\gamma_1+\gamma_2)}{d^2} \begin{pmatrix}
 1 \\
 -1
\end{pmatrix}\textnormal{Im}\big\{{w}\big\}
\\ &-\begin{pmatrix}
{\gamma_1}\textnormal{Im}\Big\{h'_1(w)\Big\} \\
{\gamma_2}\,\textnormal{Im}\Big\{h'_2(w)\Big\}
\end{pmatrix}.
\notag
\end{align*}
which is invertible from $\mathbb{R}\times\mathbb{R}\times X$ to $Y$. 

 This allows to to apply the Implicit Function Theorem and complete the proof of the main theorem.
 
 \vspace{0,2cm}

{\bf {Notation.}}
We need to fix some notation that will  be frequently used along this paper.
 We shall use the symbol $\triangleq$, to define an object.
 We denote by $\mathbb{D}$ the unit disc and its boundary, the unit circle, is denoted by  $\mathbb{T}$. 
 Let $f:\mathbb{T}\to \CC$ be a continuous function. We define  its  mean value by,
$$
\fint_{\mathbb{T}} f(\tau)d\tau\triangleq \frac{1}{2i\pi}\int_{\mathbb{T}}  f(\tau)d\tau,
$$
where $d\tau$ stands for the complex integration.
\section{Boundary equation}\label{Sec-bound}
In what follows, we shall describe the motion of a co-rotating and a counter-rotating asymmetric pair of patches  in the plane.
  The equations governing the  dynamics of the boundaries   can be thought of as a system  of two steady nonlocal
equations of nonlinear type coupling the conformal parametrization  of the two domains.

\subsection{ Co-rotating vortex pairs }
Let $\varepsilon\in(0,1)$ and consider  two bounded simply connected domains $D_1^\varepsilon$ and $D_2^\varepsilon$, containing the origin and contained in the ball $B(0,2)$. 
For $b_1, b_2\in\mathbb{R}_+$ and $d>2(b_1+b_2)$  we define the domains
\begin{equation}\label{domains}
 \tilde{D}_1^\varepsilon\triangleq\varepsilon b _1 D_1^\varepsilon\quad \textnormal{and}\quad \tilde{D}_2^{\varepsilon}\triangleq-\varepsilon b_2  D_2^\varepsilon+d.
\end{equation}
Given $\gamma_1,\gamma_2\in \mathbb{R}$, we consider the initial vorticity 
\begin{equation}\label{omega0eps}
\omega_{0,\varepsilon}=\frac{\gamma_1}{\varepsilon^2b_1^2}\chi_{\tilde{D}_1^\varepsilon}+\frac{\gamma_2}{\varepsilon^2b_2^2}\chi_{\tilde{D}_2^{\varepsilon}}.
\end{equation}
 As we can readily observe, this initial data is composed
of unequal-sized  pair of simply connected patches with  vorticity magnitudes $\gamma_1$ and $\gamma_2$. 
Assume that the initial vorticity $\omega_{0,\varepsilon}$ gives rise to a rotating pair of patches about some point $Z$ in the complex plan. Since the boundary of $\tilde{D}_1^\varepsilon \cup \tilde{D}_2^\varepsilon$ is  advected by the flow, then it is folklore (see,for instance,
 \cite{HFMV}) that this condition can be expressed by the equation
\begin{equation}\label{eq:rot}
\textnormal{Re}\bigg(-i\Omega\Big(\overline{z}-Z\Big)\vec{n}\bigg)=\textnormal{Re}\Big(\overline{v(z)}\vec{n}\Big)\quad \forall z\in \partial \tilde{D}_1^\varepsilon \cup \partial \tilde{D}_2^{\varepsilon},
\end{equation}
where $\vec{n}$ is the exterior unit normal vector to the boundary 
at the point $z$.
From the  Biot-Savart law, the velocity can be recovered from the vorticity by 
\begin{align*}
\overline{v(z)} &=\frac{i\gamma_1}{2\pi \varepsilon^2b_1^2}\int_{\tilde{D}_1^\varepsilon}\frac{dA(\zeta)}{\zeta-z}+\frac{i\gamma_2}{2\pi \varepsilon^2 b_2^2}\int_{\tilde{D}_2^{\varepsilon}}\frac{dA(\zeta)}{\zeta-z}, \quad\forall z\in \CC. 
\end{align*}
Therefore, by using  Green-Stokes formula we get 
\begin{align*}
\overline{v(z)} &=\frac{i\gamma_1}{2\varepsilon^2b_1^2}\fint_{\partial \tilde{D}_1^\varepsilon}\frac{\overline{\xi}-\overline{z}}{\xi-z}d\xi+\frac{i\gamma_2}{2\varepsilon^2 b_2^2}\fint_{\partial \tilde{D}_2^\varepsilon}\frac{\overline{\xi}-\overline{z}}{\xi-z}d\xi, \quad\forall z\in \CC. 
\end{align*}
Changing  $\xi$ by $ - \xi+d$ 
  in the last integral gives
  \begin{equation}\label{eq2}
\overline{v(z)} =\frac{i\gamma_1}{2\varepsilon^2b_1}\fint_{\partial \tilde{D}_1^\varepsilon}\frac{ \overline{\xi}-\overline{z}}{ \xi-z}d\xi-\frac{i\gamma_2}{2\varepsilon^2 b_2^2}\fint_{ \partial  \tilde{D}_3^\varepsilon}\frac{ \overline{\xi}+\overline{z}-d}{ \xi+z-d}d\xi
\end{equation}
where we have used the notation 
$$\tilde{D}_3^\varepsilon\triangleq -\tilde{D}_2^\varepsilon+d =  \varepsilon b_2 D_2^\varepsilon.$$
Hence by combining the last identity with \eqref{eq2} and \eqref{eq:rot} we obtain
\begin{align*}
&\displaystyle\textnormal{Re}\bigg\{\bigg(2\Omega\big(\overline{z}-Z\big)+\frac{\gamma_1}{\varepsilon^2b_1^2}\fint_{\partial \tilde{D}_1^\varepsilon}\frac{\overline{\xi}-\overline{z}}{\xi-z}d\xi-\frac{\gamma_2}{\varepsilon^2 b_2^2}\fint_{\partial \tilde{D}_3^\varepsilon}\frac{\overline{\xi}+\overline{z}-d}{\xi+z-d}d\xi\bigg)\vec{\tau}\bigg\}=0\quad \forall z\in \partial \tilde{D}_1^\varepsilon , \\
&\displaystyle\textnormal{Re}\bigg\{\bigg(2\Omega\big(\overline{z}+Z-d\big)-\frac{\gamma_1}{\varepsilon^2b_1^2}\fint_{\partial \tilde{D}_1^\varepsilon}\frac{\overline{\xi}+\overline{z}-d}{\xi+z-d}d\xi+\frac{\gamma_2}{\varepsilon^2 b_2^2}\fint_{\partial \tilde{D}_3^\varepsilon}\frac{\overline{\xi}-\overline{z}}{\xi-z}d\xi\bigg)\vec{\tau}\bigg\}=0\quad \forall z\in \partial \tilde{D}_3^\varepsilon 
\end{align*}
where $\vec{\tau}$ denotes  a tangent vector to the boundary of $\tilde{D}_1^\varepsilon \cup \tilde{D}_3^\varepsilon$ at the point $z$.
 Observe that when $z\in \partial \tilde{D}_1^\varepsilon$ then  $- z+{d} \not\in \overline{\tilde{D}_3^\varepsilon}$ and conversely; when $z\in \partial \tilde{D}_3^\varepsilon$, then  $-z+d\not\in \overline{\tilde{D}_1^\varepsilon}$.  Hence  according  to  residue theorem, we find that for  $j\in\{1,3\}$,
 $$
\fint_{\partial \tilde{D}_j^\varepsilon}\frac{\overline{\xi}+\overline{z}-d}{\xi+z-d}d\xi=\fint_{\partial \tilde{D}_j^\varepsilon}\frac{\overline{\xi}}{\xi+z-d}d\xi  \quad \forall z\in \partial \tilde{D}_{3-j}^\varepsilon\cdot
$$
Combining this identity with  the change of variables $z\to \varepsilon b_1 z$ and $z\to \varepsilon b_2z $,  which send $\tilde{D}_1^\varepsilon$ to $D_1^\varepsilon$ and $\tilde{D}_3^\varepsilon$ to $D_2^\varepsilon$ respectively, we get
\begin{align*}
&\textnormal{Re}\bigg\{\displaystyle\bigg(2\Omega\big(\varepsilon\overline{z}-Z\big)+\displaystyle\frac{\gamma_1}{\varepsilon b_1}\fint_{\partial D_1^\varepsilon}\frac{\overline{\xi}-\overline{z}}{\xi-z}d\xi-{\gamma_2}\fint_{\partial D_2^\varepsilon}\frac{\overline{\xi}}{\varepsilon b_2 \xi+\varepsilon b_1 z-d}d\xi\bigg)\vec{\tau}\bigg\}=0\quad \forall z\in \partial D_1^\varepsilon , \\
&\textnormal{Re}\bigg\{\displaystyle\bigg(2\Omega\big(\varepsilon b\overline{z}+Z-d\big)-\gamma_1\displaystyle\fint_{\partial D_1^\varepsilon}\frac{\overline{\xi}}{\varepsilon b_1 \xi+\varepsilon b_2 z-d}d\xi+\frac{\gamma_2}{\varepsilon b_2}\fint_{\partial D_2^\varepsilon}\frac{\overline{\xi}-\overline{z}}{\xi-z}d\xi\bigg)\vec{\tau}\bigg\}=0\quad \forall z\in  \partial D_2^\varepsilon.
\end{align*}
The structure of the last system may be unified as follows
\begin{align}\label{eq5}
\textnormal{Re}\bigg\{\bigg(2\Omega \Big(\varepsilon b_j\overline{z}&+(-1)^j\overline{z}_0-(j-1)d\Big)+\displaystyle\frac{\gamma_j}{\varepsilon b_j}\fint_{\partial D_j^\varepsilon}\frac{\overline{\xi}-\overline{z}}{\xi-z}d\xi\\ &-{\gamma_{3-j}}\fint_{\partial D_{3-j}^\varepsilon}\frac{\overline{\xi}}{\varepsilon b_{3-j}\xi+\varepsilon b_j  z-d}d\xi\bigg)\vec{\tau}\bigg\}=0, \; \forall z\in \partial D_j, \quad j=1,2.\notag
\end{align}
We shall look for domains $\{D_j,j=1,2\}$ which are slight perturbation of the unit disc of order $\varepsilon$. In other words, we shall impose the conformal mapping $\phi_j:\mathbb{D}^c\to [D_j^\varepsilon ]^c$ to satisfy the expansions
$$
\phi_j(w)= w+\varepsilon b_j f_j(w), \quad \textnormal{with}\quad f_j(w)=\sum_{n\geq 1} \frac{a_n^j}{w^n}, \quad a_n^j\in \RR, j=1,2.
$$
Because a tangent vector to the boundary at  $z=\phi_j(w)$ is determined by
$$
\vec{\tau}\big(\phi_j(w)\big)=iw\phi'_j(w),
$$
 then by a change of variables the steady vortex pairs system \eqref{eq5} becomes: for all  $w\in \mathbb{T}$, 
\begin{equation}\label{eq:steady}
\textnormal{Im}\bigg\{\bigg(2\Omega\Big(\varepsilon b_j\overline{\phi_j(w)}+(-1)^jZ-(j-1)d\Big)+{\gamma_{j}}I_j^\varepsilon(w)-\triangleq{\gamma_i}K_{3-jj}^\varepsilon(w)\bigg)w\phi'_j(w)\bigg\}=0,\quad j=1,2
\end{equation}
 where
\begin{align*}
I_j^\varepsilon(w)&\triangleq\frac{1}{\varepsilon b_j}\fint_{\mathbb{T}}\frac{\overline{\phi_j(\tau)}-\overline{\phi_j(w)}}{\phi_j(\tau)-\phi_j(w)}\phi'_j(\tau)d\tau
\end{align*}
and
\begin{align*}
K_{ij}^\varepsilon(w)&\displaystyle\fint_{\mathbb{T}}\frac{\overline{\phi_i(\tau)}\phi'_i(\tau)}{\varepsilon b_i \phi_i(\tau)+\varepsilon b_j\phi_j(w)- d}d\tau.
\end{align*}
Now, we shall see how to remove the singularity in $\varepsilon$  from the full non linearity $I_j^\varepsilon(w)$. To alleviate the discussion, we use the notation
$$A=\tau-w\quad\textnormal{ and}\quad B_j=f_j(\tau)-f_j(w).$$
Thus
\begin{align*}
I_j^\varepsilon(w)&=\frac{1}{\varepsilon b_j}\fint_{\mathbb{T}}\frac{\overline{A}+\varepsilon b_j\overline{B_j}}{A+\varepsilon b_j B_j}\Big[1+\varepsilon b_j f_j'(\tau)\Big]d\tau
\\ &=\fint_{\mathbb{T}}\frac{\overline{A}}{A}f'_j(\tau)d\tau+\varepsilon b_j\fint_{\mathbb{T}}\frac{A\overline{B}_j-\overline{A}B_j}{A(A+\varepsilon  b_jB_j)}f'_j(\tau)d\tau+\fint_{\mathbb{T}}\frac{A\overline{B}_j-\overline{A}B_j}{A^2}d\tau\\ & -\varepsilon b_j\fint_{\mathbb{T}}\frac{\big(A\overline{B}_j-\overline{A}B_j\big)B_j}{A^2(A+\varepsilon b_j B_j)}d\tau+\frac{1}{\varepsilon b_j}\fint_{\mathbb{T}}\frac{\overline{A}}{A}d\tau\cdot
\end{align*}
From the identity 
$$
\fint_{\mathbb{T}}\frac{\overline{A}}{A}d\tau=\fint_{\mathbb{T}}\frac{\overline{w}-\overline{\tau}}{w-\tau}d\tau=-\overline{w}
$$
and by the residue theorem
$$
\fint_{\mathbb{T}}\frac{\overline{A}}{A}f'_j(\tau)d\tau=0,\quad\fint_{\mathbb{T}}\frac{A\overline{B}_j-\overline{A}B_j}{A^2}d\tau=0.
$$ 
Therefore
\begin{align}\label{eqides}
\textnormal{Im}\Big\{I_j^\varepsilon(w)w\phi'_j(w)\Big\}&=\varepsilon b_j
\textnormal{Im}\bigg\{w\Big(1+\varepsilon b_j f'_j(w)\Big)\fint_{\mathbb{T}}\frac{A\overline{B}_j-\overline{A}B_j}{A(A+\varepsilon  b_jB_j)}\Big[f'_j(\tau)-\frac{B_j}{A}\Big]d\tau \bigg\}\\ & - \textnormal{Im}\Big\{f'_j(w)\Big\} .\notag
\end{align}
Inserting  the last identity  in the system \eqref{eq:steady} we get 
\begin{align}
F_j(\varepsilon,g)(w)=0\quad \forall w\in \mathbb{T}
\end{align}
where
$$
g\triangleq (\Omega,Z,f_1,f_2)
$$
and
\begin{align}\label{Fj}
F_j(\varepsilon,g)&\triangleq\textnormal{Im}\Bigg\{2\Omega\bigg[\varepsilon b_j\big(\overline{w}+\varepsilon  b_j\overline{f_j(w)}\big)+(-1)^j{Z}-(j-1)d\bigg]w\Big(1+\varepsilon b_j f'_j(w)\Big)
-{\gamma_j} f'_j(w)\\ &+\varepsilon b_j {\gamma_j} w\Big(1+\varepsilon b_j f'_j(w)\Big)\fint_{\mathbb{T}}\frac{A\overline{B}_j-\overline{A}B_j}{A(A+\varepsilon b_j B_j)}\Big[f'_j(\tau)-\frac{B_j}{A}\Big]d\tau
\notag\\ & -\gamma_{3-j}w\Big(1+\varepsilon  b_jf'_j(w)\Big)\displaystyle\fint_{\mathbb{T}}\frac{\big(\overline{\tau}+\varepsilon b_{3-j} \overline{ f_{3-j}(\tau)}\big)\big(1+\varepsilon b_{3-j} f'_{3-j}(\tau)\big)}{\varepsilon \big(b_{3-j}\tau+b_j w\big)+\varepsilon^2\big( b_{3-j}^2 f_{3-j}(\tau)+b_j^2 f_j(w)\big)- d}d\tau\notag\Bigg\}\\ & \triangleq\textnormal{Im}\Bigg\{F_{1j}(\varepsilon,g)(w)+F_{2j}(\varepsilon,g)(w)+F_{3j}(\varepsilon,g)(w)\notag\Bigg\}\
\end{align}
Next we use the following notation
\begin{equation}\label{eq:G}
F(\varepsilon,g)(w)\triangleq \Big(F_1(\varepsilon,g)(w),F_2(\varepsilon,g)(w)\Big)\quad \forall w\in \mathbb{T}\cdot
\end{equation}

\subsection{Counter-rotating vortex pair }
We shall consider  two bounded simply connected domains $D_1^\varepsilon$ and $D_2^\varepsilon$, containing the origin and contained in the ball $B(0,2)$, . 
For $b_1, b_2\in(0,\infty)$ and $d>b_1+b_2$  we set
\begin{equation}\label{domains2}
\tilde{D}_1^\varepsilon\triangleq\varepsilon b _1 D_1^\varepsilon\quad \textnormal{and}\quad \tilde{D}_2^{\varepsilon}\triangleq-\varepsilon b_2  D_2^\varepsilon+d.
\end{equation}
Given $\gamma_1,\gamma_2\in \mathbb{R}$, we consider the initial vorticity 
\begin{equation}\label{omega0eps2}
\omega_{0,\varepsilon}=\frac{\gamma_1}{\varepsilon^2b_1^2}\chi_{\tilde{D}_1^\varepsilon}-\frac{\gamma_2}{\varepsilon^2b_2^2}\chi_{\tilde{D}_2^{\varepsilon}}.
\end{equation}
We assume that this unequal-sized  pair of simply connected patches with  vorticity magnitudes $\gamma_1$ and $-\gamma_2$ travels steadily in $(Oy)$ direction with uniform velocity $U$. Then in the moving frame the pair of the patches is stationary and consequently, =
\begin{equation}\label{eq:rot2}
\textnormal{Re}\Big\{\big(\overline{v(z)}+iU\big)\vec{n}\Big\}\quad \forall z\in \partial \tilde{D}_1^\varepsilon \cup \partial \tilde{D}_2^{\varepsilon},
\end{equation}
where $\vec{n}$ is the exterior unit normal vector to the boundary of $ \tilde{D}_1^\varepsilon \cup \tilde{D}_2^\varepsilon$ at the point $z$.
From \eqref{eq2} one has
\begin{align*}
&\displaystyle\textnormal{Re}\bigg\{\bigg(2U+\frac{\gamma_1}{\varepsilon^2b_1^2}\fint_{\partial \tilde{D}_1^\varepsilon}\frac{\overline{\xi}-\overline{z}}{\xi-z}d\xi+\frac{\gamma_2}{\varepsilon^2 b_2^2}\fint_{\partial  \tilde{D}_3^\varepsilon}\frac{\overline{\xi}+\overline{z}-d}{\xi+z-d}d\xi\bigg)\vec{\tau}\bigg\}=0\quad \forall z\in \partial \tilde{D}_1^\varepsilon , \\
&\displaystyle\textnormal{Re}\bigg\{\bigg(2U+\frac{\gamma_1}{\varepsilon^2b_1^2}\fint_{\partial \tilde{D}_1^\varepsilon}\frac{\overline{\xi}+\overline{z}-d}{\xi+z-d}d\xi+\frac{\gamma_2}{\varepsilon^2 b_2^2}\fint_{\partial  \tilde{D}_3^\varepsilon}\frac{\overline{\xi}-\overline{z}}{\xi-z}d\xi\bigg)\vec{\tau}\bigg\}=0\quad \forall z\in \partial \tilde{D}_3^\varepsilon 
\end{align*}
where we have used the notation $$\tilde{D}_3^\varepsilon\triangleq  \varepsilon b_2 D_2^\varepsilon.$$
and $\vec{\tau}$ denote for  tangent vector to the boundary of $ \tilde{D}_1^\varepsilon \cup \tilde{D}_3^\varepsilon$ at the point $z$.
We can easily verify that when $z\in \partial \tilde{D}_1^\varepsilon$, $- z+{d} \not\in \overline{\tilde{D}_3^\varepsilon}$. Therefore, if $z\in \partial \tilde{D}_3^\varepsilon$ then  $-z+d\not\in \overline{\tilde{D}_1^\varepsilon}$. Thus, by  the residue theorem we conclude that for all $i,j\in\{1,3\}$ and $i\neq j$, we have
$$
\fint_{\partial  \tilde{D}_j^\varepsilon}\frac{\overline{\xi}+\overline{z}-d}{\xi+z-d}d\xi=\fint_{\partial  \tilde{D}_j^\varepsilon}\frac{\overline{\xi}}{\xi+z-d}d\xi  \quad \forall z\in \partial \tilde{D}_{3-j}^\varepsilon\cdot
$$
Inserting the last identity into the system above and changing  $z\to \varepsilon b_1 z$ and $z\to \varepsilon b_2z $ we find
\begin{equation}\label{eq52}
\textnormal{Re}\bigg\{\bigg(2U+\displaystyle\frac{\gamma_j}{\varepsilon b_j}\fint_{\partial D_j^\varepsilon}\frac{\overline{\xi}-\overline{z}}{\xi-z}d\xi+{\gamma_{3-j}}\fint_{\partial D_{3-j}^\varepsilon}\frac{\overline{\xi}}{\varepsilon b_{3-j}\xi+\varepsilon b_j  z-d}d\xi\bigg)\vec{\tau}\bigg\}=0, \; \forall z\in \partial D_j
\end{equation}
with $j=1,2$. 
We shall now use the conformal parametrization of the boundaries $\phi_j:\mathbb{D}^c\to [D_j^\varepsilon ]^c$,
$$
\phi_j(w)= w+\varepsilon b_j f_j(w), \quad \textnormal{with}\quad f_j(w)=\sum_{n\geq 1} \frac{a_n^j}{w^n}, \quad a_n^j\in \RR, j=1,2.
$$
Since the tangent vector to the boundary at  $z=\phi_j(w)$ is given by
$$
\vec{\tau}\big(\phi_j(w)\big)=iw\phi'_j(w),
$$
 then by a change of variables the  system \eqref{eq52} becomes
\begin{equation}\label{eq:steady2}
\textnormal{Im}\bigg\{\bigg(2U+{\gamma_j}I_j^\varepsilon(w)+{\gamma_{3-j}}K_{j-3j}^\varepsilon(w)\bigg)w\phi'_j(w)\bigg\}=0,\quad j=1,2
\end{equation}
 for all $w\in \mathbb{T}$, where
\begin{align*}
I_j^\varepsilon(w)&\triangleq\frac{1}{\varepsilon b_j}\fint_{\mathbb{T}}\frac{\overline{\phi_j(\tau)}-\overline{\phi_j(w)}}{\phi_j(\tau)-\phi_j(w)}\phi'_j(\tau)d\tau,
\end{align*}
\begin{align*}
K_{ij}^\varepsilon(w)&\triangleq\displaystyle\fint_{\mathbb{T}}\frac{\overline{\phi_{i}(\tau)}\phi'_i(\tau)}{\varepsilon b_i \phi_i(\tau)+\varepsilon b_j\phi_j(w)- d}d\tau.
\end{align*}
As in the rotating case, from \eqref{eqides}, one has 
\begin{align*}
\textnormal{Im}\Big\{I_j^\varepsilon(w)w\phi'_j(w)\Big\}&=\varepsilon b_j
\textnormal{Im}\bigg\{w\Big(1+\varepsilon b_j f'_j(w)\Big) \fint_{\mathbb{T}}\frac{A\overline{B}_j-\overline{A}B_j}{A(A+\varepsilon  b_jB_j)}\Big[f'_j(\tau)-\frac{B_j}{A}\Big]d\tau \bigg\}- \textnormal{Im}\Big\{f'_j(w)\Big\},
\end{align*}
where we have  used the notation
$$A=\tau-w\quad\textnormal{ and}\quad B_j=f_j(\tau)-f_j(w).$$
Inserting  the last identity  into  \eqref{eq:steady2} we get 
\begin{equation}
F(\varepsilon,g)(w)\triangleq\big(F_1(\varepsilon,g)(w),F_2(\varepsilon,g)(w)\big) =0,\quad\forall w\in\mathbb{T},\quad  j=1,2
\end{equation}
$$
g\triangleq (U,\gamma_2,f_1,f_2)
$$
and
  for all $w\in\mathbb{T}$ with $j=1,2$.
  \begin{align}\label{fj2}
&F_j(\varepsilon,g)\triangleq\textnormal{Im}\Bigg\{2Uw\Big(1+\varepsilon b_j f'_j(w)\Big)-{\gamma_j} f'_j(w)
\\ &+\varepsilon b_j {\gamma_j} w\Big(1+\varepsilon b_j f'_j(w)\Big)\fint_{\mathbb{T}}\frac{A\overline{B}_j-\overline{A}B_j}{A(A+\varepsilon b_j B_j)}\Big[f'_j(\tau)-\frac{B_j}{A}\Big]d\tau\notag\\ &+\gamma_{3-j}w\Big(1+\varepsilon b_j f'_j(w)\Big)\displaystyle\fint_{\mathbb{T}}\frac{\big(\overline{\tau}+\varepsilon b_{3-j} \overline{ f_{3-j}(\tau)}\big)\big(1+\varepsilon f'_{3-j}(\tau)\big)}{\varepsilon \big(b_{3-j}\tau+b_j w\big)+\varepsilon^2\big( b_{3-j}^2 f_{3-j}(\tau)+b_j^2 f_j(w)\big)- d}d\tau\notag\Bigg\}
\\&\triangleq\textnormal{Im}\Big\{F_{1j}(\varepsilon,g)+F_{2j}(\varepsilon,g)+F_{3j}(\varepsilon,g)\Big\}.\notag
\end{align}

\section{Regularity of the nonlinear functional}
This section is devoted to the regularity study of the nonlinear functional $F$ introduced
in \eqref{eq:G} and \eqref{fj2} and which defines the V-states equations. We proceed first with the  Banach spaces $X$ and $Y$ of H\"older type  to which the Implicit Function Theorem will be
applied.
Recall that for given $\alpha\in(0,1)$, we denote by $C^\alpha $ the space of continuous functions $f:\mathbb{T}\to \mathbb{C}$ such that
$$
\|f\|_{C^\alpha(\mathbb{T})}\triangleq \|f\|_{L^\infty(\mathbb{T})}+\textnormal{sup}_{\tau\neq w\in \mathbb{T}}\frac{|f(\tau)-f(w)|}{|\tau-w|^\alpha}<\infty\cdot
$$ 
For any integer $n\in \mathbb{N}$, the space $C^{n+\alpha}(\mathbb{T})$ stands for the set of functions $f$ of class $C^n$ whose $n-$th order derivative is H\"older continuous with exponent $\alpha$.  This space is equipped with
the usual norm
$$
\|f\|_{C^{\alpha+n}(\mathbb{T})}\triangleq \|f\|_{L^\infty(\mathbb{T})}+\Big\Vert \frac{d^n f}{dw^n}\Big\Vert_{C^\alpha(\mathbb{T})} \cdot
$$ 
Now, consider the spaces
$$
X\triangleq\bigg\{f\in \big( C^{1+\alpha}(\mathbb{T})\big)^2,\; f(w)=\sum_{n\geq 1}A_n\overline{w}^n,\; A_n\in \RR^2, w\in \mathbb{T}\bigg\},
$$
$$
Y=\bigg\{g\in \big( C^{\alpha}(\mathbb{T})\big)^2,\; g(w)=\sum_{n\geq 1}C_n e_n,\; C_n\in \RR^2, w\in \mathbb{T}\bigg\},\quad e_n=\textnormal{Im}\big\{{w}^n\big\},
$$
\begin{equation}\label{spaceZ}
\tilde{Y}\triangleq \bigg\{g\in Y\;:\; C_1=\begin{pmatrix}
0 \\
0 
\end{pmatrix}\bigg\}.
\end{equation}
For $r > 0$ we denote by $B_r$ the open ball of $X$ centered at zero and of radius $r$,
$$
B_r\triangleq\bigg\{f\in X,\quad ||f||_{C^{\alpha+n}(\mathbb{T})}\leq r\bigg\}\cdot
$$
It is straightforward that for any $f\in B_r$ the function $w\mapsto \phi(w) = w +\varepsilon f(w)$ is conformal
on $\mathbb{C}\backslash\overline{\mathbb{D}}$ provided that $r,\varepsilon  < 1$. Moreover according to Kellog-Warshawski result, the
boundary of $\phi(\mathbb{C}\backslash\overline{\mathbb{D}})$ is a Jordan curve of class $C^{
n+\alpha}$. 
\subsection{Co-rotating vortex pairs}
We propose to prove the following result concerning the regularity of $F$.

\begin{proposition}\label{prop1}
The following assertions hold true.
\begin{enumerate}
\item The function $F$ 
can be extended to $C^1$ function from $\big(-\frac12,\frac12\big)\times \RR\times\RR \times B_1\to Y$.
\item Two initial point vortex $\gamma_1\pi \delta_0$ and $\gamma_2\pi \delta_{d}$ with $\gamma_1\neq-\gamma_2$ rotate uniformly about the point $$Z_0\triangleq\frac{ d\gamma_2}{\gamma_1+\gamma_2}$$ with the angular velocity
\begin{equation*}
\Omega_0\triangleq\frac{\gamma_1+\gamma_2}{2d^2}\cdot
\end{equation*}
\item For all $h=(\alpha_1,\alpha_2,h_1,h_2)\in \mathbb{R}\times\mathbb{R}\times X$ one has
\begin{align*}
D_{g}F (0,g_0)h(w)=&-\dfrac{2\alpha_1 d}{\gamma_1+\gamma_2} \begin{pmatrix}
\gamma_2  \\
\gamma_1
\end{pmatrix}\textnormal{Im}\big\{{w}\big\}- \dfrac{\alpha_2 (\gamma_1+\gamma_2)}{d^2} \begin{pmatrix}
 1 \\
 -1
\end{pmatrix}\textnormal{Im}\big\{{w}\big\}-\begin{pmatrix}
{\gamma_1}\textnormal{Im}\Big\{h'_1(w)\Big\} \\
{\gamma_2}\,\textnormal{Im}\Big\{h'_2(w)\Big\}
\end{pmatrix}
\notag
\end{align*}
with
$$
g\triangleq(\Omega,Z,f_1,f_2)\quad \textnormal{and}\quad g_0\triangleq(\Omega_0,Z_0,0,0).
$$
\item  The linear operator 
$D_{g}F (0,g_0): \mathbb{R}\times\mathbb{ R} \times X\to  Y$ is an isomorphism.
\end{enumerate}

\end{proposition}
\begin{proof}
{\bf (1)} Notice that  the part $F_{1j}+F_{2j}$ defined in \eqref{Fj}  appears identically  in the boundary equation of co-rotating  symmetric pairs  and its  regularity was   discussed in the paper \cite{HM}.  The only term that one should care about is $F_{3j}$ describing the interaction between the boundaries of the two patches which are supposed to be disjoint. Therefore the  involved kernel is sufficiently smooth and it does not carry significant difficulties in the treatment. and can be dealt with in a very classical way. \\\\
{\bf(2)} According to the formulation  developed in Section \ref{Sec-bound}  the two points vortex system is a solution  to the equation $F(0,\Omega,Z,0,0)= 0$. In this case
we can easily check that 
\begin{align*}
F_j(0,\Omega,Z,0,0)(w)&=\bigg[2\Omega\Big((-1)^jZ+(1-j)d\Big)+\frac{\gamma_{3-j}}{d}\bigg]\textnormal{Im}\big\{w\big\}.
\end{align*}
Therefore $F_j(0,\Omega,Z,0,0)=0$ if and only if
$$
-2\Omega Z+\frac{\gamma_2}{d}=0\quad\textnormal{and}\quad 2\Omega\Big(Z-d\Big)+\frac{\gamma_1}{d}=0.
$$
Thus,
\begin{equation}\label{omz0}
\Omega=\Omega_0\triangleq\frac{\gamma_1+\gamma_2}{2d^2}\quad\textnormal{and}\quad Z=Z_0\triangleq\frac{d\gamma_2}{\gamma_1+\gamma_2}\cdot
\end{equation}
{\bf(3)-(4)} Since $F = (F_1, F_2)$ then for  given $h=(h_1,h_2)\in X$, we have
\begin{align*}
D_{f}F(0,\Omega,Z,0,0)h(w)&=\begin{pmatrix}
\partial_{f_1}F_1(0,\Omega,Z,0,0)h_1(w)+\partial_{2}F_1(0,\Omega,Z,0,0)h_2(w) \\
\partial_{f_1}F_2(0,\Omega,Z,0,0)h_1(w)+\partial_{f_2}F_2(0,\Omega,Z,0,0)h_2(w)
\end{pmatrix}.
\end{align*}
 The G\^ateaux derivative  of the function $F_j$ at $(0,\Omega,b,0,0)$ in the direction $h_j$ is
given by 
\begin{align*}
\partial_{f_j}F_j(0,\Omega,Z,0,0)h_j(w)&=\frac{d}{dt}\Big[F_j(0,\Omega,Z,th_j,0)\Big]_{t=0}(w)
\\ &=-{\gamma_j}\textnormal{Im}\Big\{h'_j(w)\Big\}.
\end{align*}
On the other hand, straightforward  computations allow to get
 \begin{alignat*}{2}
\partial_{f_{3-j}}F_j(0,\Omega,Z,0,0)h_{3-j}(w)&=\frac{d}{dt}\Big[F_j&&(0,\Omega,Z,0,th_{3-j})\Big]_{t=0}(w)\\&=0.
\end{alignat*}
From \eqref{Fj} one can easily check that
\begin{align}\label{feps0}
F_j(0,\Omega,Z,f_1,f_2)=\textnormal{Im}\bigg\{&2\Omega \Big((-1)^j
Z+(1-j){d}\Big)w-{\gamma_j}f'_j(w)+\frac{\gamma_{3-j}}{d}w\bigg\}.
\end{align}
Differentiating the last identity with respect to $\Omega$ at $(\Omega_0,Z_0)$ gives
\begin{align}\label{feps0omega}
\partial_{\Omega}F_j(0,\Omega_0,Z_0,f_1,f_2)&=2\Big((-1)^j
Z_0+(1-j)d\Big)\,\textnormal{Im}\big\{w\big\}.
\end{align}
Next, differentiating \eqref{feps0} with respect to $Z$ at $(\Omega_0,Z_0)$ we get
\begin{align*}
\partial_{Z} F_j(0,\Omega_0,Z_0,f_1,f_2)&=2(-1)^j\Omega_0\,\textnormal{Im}\big\{
w\big\}.
\end{align*}
Therefore, for all $(\alpha_1,\alpha_2)\in \mathbb{R}^2$ 
\begin{align*}
D_{(\Omega,Z)}F(0,\Omega_0,Z_0,f_1,f_2)(\alpha_1,\alpha_2)
=\alpha_1\begin{pmatrix}
\dfrac{2\gamma_2 d}{\gamma_1+\gamma_2} \\
\dfrac{2\gamma_1 d}{\gamma_1+\gamma_2}
\end{pmatrix}\textnormal{Im}\big\{\overline{w}\big\}+\alpha_2\begin{pmatrix}
 \dfrac{\gamma_1+\gamma_2}{d^2}  \\
 -\dfrac{\gamma_1+\gamma_2}{d^2}
\end{pmatrix}\textnormal{Im}\big\{\overline{w}\big\}.
\end{align*}
Consequently  the restricted linear operator $D_{(\Omega,Z,f_1,f_2)}F(0,\Omega_0,Z_0,f_1,f_2)$ is invertible if  if and only if the determinant of the two vectors, given by $-\displaystyle\frac2d({\gamma_1+\gamma_2})$, is non vanishing.  
\end{proof}

%
%
%
\subsection{Counter-rotating vortex pairs}
This section is devoted to the counter-rotating asymmetric  pairs and our main result reads as follows.
\begin{proposition}\label{prop12}
The following assertions hold true.
\begin{enumerate}
\item The function $F$  can be extended to $C^1$ function from $\big(-\frac12,\frac12\big)\times \mathbb{R}\times\mathbb{R}\times  B_1\to Y$.
\item Two initial point vortex $\gamma_1\pi \delta_0$ and $-\gamma_2\pi \delta_{(d,0)}$ translate uniformly  with the speed 
$$U_0\triangleq \frac{\gamma_1}{2d}\cdot
$$
\item For all $h=(\alpha_1,\alpha_2,h_1,h_2)\in \mathbb{R}\times\mathbb{R}\times X$
 one has
\begin{align}
D_gF (0,g_0)h(w)=&{\alpha_1} \begin{pmatrix}
2  \\
2
\end{pmatrix}\textnormal{Im}\big\{{w}\big\}-\alpha_2 \begin{pmatrix}
 1 \\
 0
\end{pmatrix}\textnormal{Im}\big\{{w}\big\}-2Ud\begin{pmatrix}
\textnormal{Im}\Big\{h'_1(w)\Big\} \\
\textnormal{Im}\Big\{h'_2(w)\Big\}
\end{pmatrix}\notag
\end{align}
with
$$
g\triangleq(U,\gamma_2,f_1,f_2)\quad \textnormal{and}\quad g_0\triangleq (U_0,\gamma_1,0,0).
$$
\item The linear operator 
$D_gF (U_0,\gamma_1,0,0): \mathbb{R}\times\mathbb{ R} \times X\to  Y$ is an isomorphism.

\end{enumerate}

\end{proposition}
\begin{proof}
{\bf (1)}
 Notice that  the function $F_{j2}+F_{j3}$ coincide with the function $F_{j2}+F_{j3}$ appearing in the rotating case and  $F_{j1}$ enjoys abviously the required regularity.
\\
{\bf(2)} The point vortex configuration corresponds to $F(0,U, \gamma_2,0,0)$. Thus, for $\varepsilon=0$ one has 
\begin{align*}
F_j(0,\gamma_1,\gamma_2,0,0)(w)&=\bigg[{ 2U }-\frac{\gamma_{3-j}}{d}\bigg]\textnormal{Im}\big\{w\big\}.
\end{align*}
Therefore $F(0,U,\gamma_2,0,0)=0$ if and only if 
$$\gamma_1=\gamma_2=2Ud.$$
{\bf(3)-(4)} Taking $\varepsilon=0$ in \eqref{fj2} we get
\begin{align*}
F_j(0,U,\gamma_2,f_1,f_2)=\textnormal{Im}\bigg\{&2Uw-{\gamma_j}f'_j(w)-\frac{\gamma_{3-j}}{d}w\bigg\}.
\end{align*}
Thus,
for  given $(h_1,h_2)\in X$, we have
\begin{align*}
\partial_{f_j}F_j(0,U,\gamma_2,0,0)h_j=-\gamma_j\textnormal{Im}\Big\{h'_j(w)\Big\}
\end{align*}
and
 \begin{alignat*}{2}
\partial_{f_{3-j}}F_j(0,U,\gamma_2,0,0)h_{3-j}(w)&=0.
\end{alignat*}
%
On the other hand,  differential  of $F$ with respect $(U, \gamma_2)$ is given by
\begin{align*}
D_{(U,\gamma_2)}F(0,U,\gamma_2,f_1,f_2)(\alpha_1, \alpha_2)&=\alpha_1\begin{pmatrix}
2 \\
2
\end{pmatrix}\textnormal{Im}\big\{{w}\big\}+\alpha_2\begin{pmatrix}
 -\frac{1}{d}  \\
 0
\end{pmatrix}\textnormal{Im}\big\{{w}\big\}.
\end{align*}
Thus, we conclude that the linear operator $D_{g}F(0,g_0)$ is invertible for all  $\gamma_1\in \mathbb{R}$.\\

\end{proof}

\section{Existence of asymmetric co-rotating patches}
The  goal of this section is to provide a full statement of the point $(1)$ of Theorem \ref{thm:informal}. In other words, we shall describe  the set of solutions  of  the  equation 
${F}\big(\varepsilon,g\big)=0$
 around  the point $(\varepsilon, g\big)=(0,g_0)$ by a one-parameter smooth  curve  $\varepsilon\mapsto g(\varepsilon)$ using   the implicit function theorem.

\subsection{Co-rotating vortex pairs}
The main result of this section reads as follows.
\begin{proposition}\label{propasym0} 

The following assertions hold true.

\begin{enumerate}

\item There exists $\varepsilon_0>0$ and a unique $C^1$ function $g\triangleq(\Omega,Z,f_1,f_2): [-\varepsilon_0,\varepsilon_0]\longrightarrow  \mathbb{R}\times\mathbb{R}\times B_1$ such that 

\begin{equation*}\label{gf1f2}
 F\Big(\varepsilon, \Omega(\varepsilon),Z(\varepsilon),f_1(\varepsilon), f_2(\varepsilon)\Big)=0
\end{equation*}
with
\begin{equation*}
\Big(\Omega(0),Z(0),f_1(0), f_2(0)\Big)=\big(\Omega_0,Z_0,0,0\big).
\end{equation*}

\item For all $\varepsilon\in[-\varepsilon_0,\varepsilon_0]\backslash \{0\}$ one has
$$
\big(f_1(\varepsilon),f_2(\varepsilon)\big)\neq (0,0).
$$ 
\item If $(\varepsilon, f_1,f_2)$ is a solution then $(-\varepsilon, \tilde{f}_1,\tilde{f}_2)$ is also a solution, where
$$
\forall w \in \mathbb{T},\quad 
\tilde{f}_j(w) = f_j(-w), \quad j=1,2.
$$ 
\item For all $\varepsilon \in [-\varepsilon_0, \varepsilon]\backslash\{0\}$, the domains $D_1^\varepsilon$ and $D_2^\varepsilon$ are
 strictly convex.

\end{enumerate}

\end{proposition}
\begin{proof}

{\bf (1)} From Proposition \ref{prop1} one has  $F:(-\frac12,\frac12)\mathbb{R}\times\mathbb{R}\times B_1\to {Y}$ is a $C^1$ function and  $D_{g} F\big(0,g_0\big) : X \to {Y}$ is an isomorphism. Thus the implicit function theorem ensures the existence of $\varepsilon_0>0$ and a unique $C^1$ parametrization  $g=(\Omega,Z,f_1,f_2): [-\varepsilon_0,\varepsilon_0]\to B_1$ verifying  the equation  $ F\big(\varepsilon, \Omega(\varepsilon),Z(\varepsilon),f_1(\varepsilon), f_2(\varepsilon)\big)=0$. Moreover,  this solution passes through the origin, that is, 
$$\big(\Omega, Z, f_1,f_2\big)(0)=(0,0).$$ 
This  completes the proof  of the desired result.\\\\
{\bf (3)} We shall prove that for any $\varepsilon$ small enough and any $\Omega$ we can not get
a rotating  vortex pair with $f_1 = 0$ and $f_2=0$.  In other words , we need to prove that
$$
F(\varepsilon,\Omega,Z,0, 0) \neq 0.
$$
To this end we write
\begin{align*}
F_j(\varepsilon,\Omega,Z,0,0)(w) =\textnormal{Im}\Bigg\{&\bigg[2\Omega \Big((-1)^jZ-(j-1)d\Big)-\gamma_{3-j}\fint_{\mathbb{T}}\frac{\overline{\tau}}{\varepsilon \big(b_{3-j}\tau+b_j w\big)- d}d\tau\bigg]w
\Bigg\}.
\end{align*}
By Taylor expansion, we get
\begin{align}\label{taylorexpanf3}
\fint_{\mathbb{T}}\frac{\overline{\tau}}{\varepsilon \big(b_{3-j}\tau+b_j w\big)- d}d\tau&=-\sum_{n\in\mathbb{N}}\frac{\varepsilon^n}{d^{n+1}}\fint_{\mathbb{T}}\overline{\tau}\big(b_{3-j}\tau+b_j w\big)^nd\tau\\&=-\sum_{n\in\mathbb{N}}\frac{\varepsilon^n}{d^{n+1}}b_j^nw^n\notag\\&=\frac{1}{\varepsilon b_j w-d}\cdot\notag
\end{align}
Consequently,
\begin{align*}
F_j(\varepsilon,\Omega,Z,0,0)(w) =\textnormal{Im}\Bigg\{&\bigg[2\Omega \Big((-1)^jZ-(j-1)d\Big)-\frac{\gamma_{3-j}}{\varepsilon b_j w-d}\bigg]w
\Bigg\}
\end{align*}
and this quantity is not zero if $\varepsilon\neq 0$ is small enough.\\

{\bf (2)} 
We only need to check that
$$
F_{j}(\varepsilon, \Omega,Z, f_1,f_2)(-w) = -F_{j}(-\varepsilon,\Omega,Z,\tilde{f}_1,\tilde{f}_2)(w)\quad \textnormal{for}\quad \quad j=1,2. 
$$
where $F_{j}$ is defined by \eqref{Fj}. By definition we have
\begin{align*}
F_{1j}(-\varepsilon,\Omega,Z,\tilde{f}_1,\tilde{f}_2)=
2\Omega\bigg[&-\varepsilon b_j\big(\overline{w}-\varepsilon b_j\tilde{f}_j(\overline{w})\big)+(-1)^jZ\\ &-(j-1)d\bigg]w\Big(1-\varepsilon b_j \tilde{f}'_j(w)\Big)
-\gamma_j \tilde{f}'_j(w).\notag
\end{align*}
Since $\tilde{f}'_j(w) = -f_j'(-w)$  we get
\begin{alignat*}{2}
F_{1j}(-\varepsilon,\Omega,Z,\tilde{f}_1,\tilde{f}_2)&=
2\Omega\bigg[&&\varepsilon b_j\big(-\overline{w}+\varepsilon b_j {f}_j(-\overline{w})\big)+(-1)^jZ\\ & &&-(j-1)d\bigg]w\Big(1+\varepsilon b_j {f}'_j(-w)\Big)
+\gamma_j f'_j(-w)\\ &=-F_{ij}&&(\varepsilon , \Omega,Z, f_1,f_2)(-w).
\end{alignat*}
Straightforward computations will lead to the same properties for the functions $F_{2j}$ and $F_{3j}$. This completes the proof of {\bf (4)}.\\

{\bf (3)}  Now to prove the convexity of the domains $D_1^\varepsilon$, $D_2^\varepsilon$ we shall reproduce the same
arguments of \cite{HM}. Recall that the outside conformal mapping associated to the domain $D^\varepsilon_j$ is
given by
$$
\phi_j = w + \varepsilon b_j
f_j(w).
$$
and the curvature can be expressed by the formula
$$
\kappa(\theta) = \frac{1}{|\phi_j^\varepsilon(w)|}
\textnormal{Re}\Big(1+w\frac{\phi_j^{''}(w)}{\phi_j'(w)} \Big).
$$
We can easily verify that
$$
1+w\frac{\phi_j^{''}(w)}{\phi_j'(w)}=1+\varepsilon b_j w\frac{f_j^{''}(w)}{1+\varepsilon b_j f_j'(w)} 
$$
and so
$$
\textnormal{Re}\Big(1+w\frac{\phi_j^{''}(w)}{\phi_j'(w)} \Big)\geq 1-|\varepsilon| b_j\frac{|f_j^{''}(w)|}{1-|\varepsilon| b_j
| f_j'(w)|}\geq \frac{|\varepsilon| b_j}{1- |\varepsilon| b_j
}\cdot
$$
The last quantity  is non-negative if $|\varepsilon| < 1/2$. Thus the curvature is strictly positive and therefore the
domain $D_j^\varepsilon$ is strictly convex.
\end{proof}
\subsection{Counter-rotating vortex pairs}
In this section we shall give the complete statement for the existence  of translating unequal-sized pairs of patches.
\begin{proposition}\label{propasym} 

The following holds true.

\begin{enumerate}
%

\item There exists $\varepsilon_0>0$ and a unique $C^1$ function $g\triangleq(U,\gamma_2,f_1,f_2): [-\varepsilon_0,\varepsilon_0]\longrightarrow  \mathbb{R}\times\mathbb{R}\times B_1$ such that 

\begin{equation*}\label{gf1f2}
 F\Big(\varepsilon, U(\varepsilon),\gamma_2(\varepsilon),f_1(\varepsilon), f_2(\varepsilon)\Big)=0,
\end{equation*}
with
\begin{equation*}
\Big(U(0),\gamma_2(0),f_1(0), f_2(0)\Big)=\Big(\frac{\gamma_1}{2d},\gamma_1,0,0\Big).
\end{equation*}

%
%
\item If $(\varepsilon, f_1,f_2)$ is a solution then $(-\varepsilon, \tilde{f}_1,\tilde{f}_2)$ is also a solution, where
$$
\forall w \in \mathbb{T},\quad 
\tilde{f}_j(w) = f_j(-w), \quad j=1,2.
$$ 
\item For all $\varepsilon \in [-\varepsilon_0, \varepsilon]\backslash\{0\}$, the domains $D_1^\varepsilon$ and $D_2^\varepsilon$ are
 strictly convex

\end{enumerate}

\end{proposition}
\begin{proof}
Same arguments as in Proposition \ref {propasym0} .
\end{proof}
\section{Asymptotic behavior}
In this section we study  the asymptotic expansion for small $\varepsilon$  of  the conformal mappings, the angular velocity, the center of rotation and the circulations. 
\subsection{Co-rotating pairs}  

In the case of the co-rotating pairs we get the following expansion at higher  order  in $\varepsilon.$
\begin{proposition}\label{asymf0}
The  conformal mappings of the co-rotating domains, $\phi_j:\mathbb{T}\to \partial D_j^\varepsilon$,   have the expansions 
\begin{align*}
\phi_j(\varepsilon,w)&=w+\delta_j\Big(\frac{\varepsilon b_j }{ d}\Big)^2\overline{w}+\frac{\delta_j}{2}\Big(\frac{\varepsilon b_j }{ d}\Big)^3\overline{w}^2+\frac{\delta_j}{3}\Big(\frac{\varepsilon b_j}{d}\Big)^4\Big(\overline{w}^3+6\big(1+\delta_j\big)\overline{w}\Big)\\ &+\frac{\delta_j}{4}\Big(\frac{\varepsilon b_j}{d}\Big)^5\Big(\overline{w}^4+3\big(1+\delta_j\big)\overline{w}^2\Big)+o(\varepsilon^5) ,
\end{align*}
for all $w\in\mathbb{T}$ and $j=1,2$ where $\delta_j=\frac{ \gamma_{3-j}}{ \gamma_j}$.
The angular velocity has the expansion
$$
\Omega(\varepsilon)=\frac{\gamma_1+\gamma_2}{2d^2}+\frac{\varepsilon^4}{2d^6}\Big(
{\gamma_1} b_2^4+{\gamma_2} b_1^4 \Big),
$$
and the center of rotation has the expansion
$$
Z(\varepsilon)=\frac{\gamma_2 d}{\gamma_1+\gamma_2}+\frac{\varepsilon^4}{d^3(\gamma_1+\gamma_2)^2}\Big(
\frac{\gamma^3_2}{\gamma_1} b_1^4-\frac{\gamma_1^3}{\gamma_2} b_2^4 \Big).
$$
\end{proposition}
\begin{proof} 
Recall that
\begin{equation*}
\phi_j(\varepsilon,w)=w+\varepsilon b_j f_j(\varepsilon,w), \quad j=1,2.
\end{equation*}
and from Proposition \ref{propasym0} one has
\begin{equation}\label{gf1f2}
F\Big(\varepsilon, g(\varepsilon)\Big)\triangleq  F\Big(\varepsilon, \Omega(\varepsilon),Z(\varepsilon),f_1(\varepsilon), f_2(\varepsilon)\Big)=0.
\end{equation}
where $F$ is defined in \eqref{Fj}. Moreover
\begin{equation}\label{f0}
g(0)=\big(\Omega_0,Z_0,0,0\big).
\end{equation} 
By the composition rule we get
\begin{equation}\label{1order0}
\partial_\varepsilon g(0)  =-D_{g} F(0,g_0)^{-1}\partial_\varepsilon F(0,g_0)\cdot
\end{equation}
In view of \eqref{Fj} one has
\begin{align*}
{F}_j\big(\varepsilon,g_0\big)&=-\gamma_{3-j}  \textnormal{Im}\bigg\{ \Big(\frac{1}{d}+ \displaystyle\fint_{\mathbb{T}}\frac{\overline{\tau}}{\varepsilon \big(b_{3-j}\tau+b_j w\big)- d}d\tau\Big)w\bigg\}
\end{align*}
and from \eqref{taylorexpanf3} we conclude that
\begin{align}\label{1f=0}
{F}_j\big(\varepsilon,g_0\big)&=\gamma_{3-j}\sum_{n=1}^\infty\frac{\varepsilon^{n}b_j^n}{d^{n+1}}  \textnormal{Im}\big\{ w^{n+1}\big\}.
\end{align}
Thus
\begin{align}\label{pepsilon}
\partial_\varepsilon^n F(0,g_0)(w)&=\frac{n!}{d^{n+1}}\begin{pmatrix}
\gamma_2 b_1^n\\
\gamma_1 b_2^n 
\end{pmatrix}\textnormal{Im}\big\{ {w}^{n+1}\big\}\cdot
\end{align}
Now, we shall write down the explicit  expression of $D_{g}F(0,g_0)^{-1}$.  Recall from Proposition  \ref{prop1} that
for all $h=(\alpha_1,\alpha_2,h_1,h_2)\in \mathbb{R}\times\mathbb{R}\times X$
with
$$
h_j(w)=\sum_{n\geq 1}a_n^j\overline{w}^{n},\quad j=1,2,
$$
 one gets
\begin{align}
D_{g}F &(0,g_0)h(w)=-\dfrac{2 \alpha_1 d}{\gamma_1+\gamma_2} \begin{pmatrix}
\gamma_2  \\
\gamma_1
\end{pmatrix}\textnormal{Im}\big\{{w}\big\}-\alpha_2  \dfrac{\gamma_1+\gamma_2}{d^2} \begin{pmatrix}
 1 \\
 -1
\end{pmatrix}\textnormal{Im}\big\{{w}\big\}-\sum_{n\geq 1}n\begin{pmatrix}
{\gamma_1}a_n^1\\
{\gamma_2}a_n^2
\end{pmatrix}\textnormal{Im}\big\{{w}^{n+1}\big\}.\notag
\end{align}
Then, for any $k\in {Y}$ with the expansion
$$
k(w)=\sum_{n\geq 0}\begin{pmatrix}
A_n \\
B_n 
\end{pmatrix}\textnormal{Im}\{{w}^{n+1}\},
$$
one has
\begin{align}\label{g-1}
D_{g}F(0,g_0)^{-1}k(w)&=\bigg(-
\dfrac{A_0+B_0 }{2d},
-\dfrac{(A_0\gamma_1-B_0\gamma_2)d^2}{(\gamma_1+\gamma_2)^2},
-\displaystyle\sum_{n\geq 1}\dfrac{A_n}{n\gamma_1}\overline{w}^{n},
-\displaystyle\sum_{n\geq 1}\dfrac{B_n}{n\gamma_2}  \overline{w}^{n}
\bigg).
\end{align}
Combining the last identity with \eqref{pepsilon} and \eqref{1order0} we get
\begin{align}\label{parepsf}
\partial_\varepsilon g(0,w) &=\bigg(
0,0,
\dfrac{b_1\gamma_2}{d^2\gamma_1}\overline{w},
\dfrac{b_2\gamma_1}{d^2\gamma_2} \overline{w}
\bigg).
\end{align}
Let's move to the second order derivative in $\varepsilon$  at $0$. By the composition rule we get 
\begin{align}\label{vareps2f}
\partial^2_{\varepsilon\varepsilon} g(0) &=-D_{g} F(0,g_0)^{-1}\Big(\partial^2_{\varepsilon} F(0,g_0)+2D^2_{\varepsilon g} F(0,g_0)g'(0)+D^2_{gg} F(0,g_0)[g'(0),g'(0)]\Big).
\end{align}
Observe  that
\begin{align}\label{dngfjn}
{F}_j(0,g)&=\textnormal{Im}\bigg\{2\Omega\Big((-1)^j Z+(1-j)d\Big) w-\gamma_jf'_j(w)+\frac{\gamma_{3-j}}{d} w\bigg\}.
\end{align}
Thus, for all $h=(\alpha_1,\alpha_2,h_1,h_2), k=(\alpha_1,\alpha_2,k_1,k_2)\in \mathbb{R}\times \mathbb{R}\times X$, one has
\begin{equation}\label{dngfj0}
\partial^2_{gg} F_j(0,g)\big[h,k\big]=(-1)^j2 \alpha_1\beta_2\textnormal{Im}\big\{w\big\}\cdot
\end{equation}
From the identity \eqref{parepsf} we get
\begin{equation}\label{dngfj}
\partial^2_{gg} F_j(0,g_0)\big[g'(0),g'(0)\big]=0\cdot
\end{equation}
Next, differentiating the identity \eqref{Fj} with respect to $\varepsilon$ we get
\begin{align}\label{varepsfj}
&\partial_\varepsilon F_j(\varepsilon,g)=\textnormal{Im}\Bigg\{2\Omega \Big[\varepsilon b_j\Big(\overline{w}+\varepsilon  b_j\overline{f_j(w)}\Big)+\Big((-1)^j{Z}-(j-1)d\Big)\Big]w b_j f'_j(w)
\\ &+2 b_j\Omega\Big(\overline{w}+2\varepsilon  b_j\overline{f_j(w)}\Big)w\Big(1+\varepsilon b_j f'_j(w)\Big)
\notag\\&+ b_j {\gamma_j} w\Big(1+2\varepsilon b_j f'_j(w)\Big)\fint_{\mathbb{T}}\frac{A\overline{B}_j-\overline{A}B_j}{A(A+\varepsilon b_j B_j)}\Big[f'_j(\tau)-\frac{B_j}{A}\Big]d\tau\notag
\\&-\varepsilon b_j^2 {\gamma_j} w\Big(1+\varepsilon b_j f'_j(w)\Big)\fint_{\mathbb{T}}\frac{B_j(A\overline{B}_j-\overline{A}B_j)}{A(A+\varepsilon b_j B_j)^2}\Big[f'_j(\tau)-\frac{B_j}{A}\Big]d\tau
\notag\\ &-\gamma_{3-j}b_j w  f'_j(w)\fint_{\mathbb{T}}\frac{\big(\overline{\tau}+\varepsilon b_{3-j} \overline{ f_{3-j}(\tau)}\big)\big(1+\varepsilon b_{3-j} f'_{3-j}(\tau)\big)}{\varepsilon C_j+\varepsilon^2D_j- d}d\tau 
\notag\\ &
-\gamma_{3-j} b_{3-j} w\Big(1+\varepsilon  b_jf'_j(w)\Big)\fint_{\mathbb{T}}\frac{\overline{ f_{3-j}(\tau)}\big(1+\varepsilon b_{3-j} f'_{3-j}(\tau)\big)+\big(\overline{\tau}+\varepsilon b_{3-j} \overline{ f_{3-j}(\tau)}\big) f'_{3-j}(\tau)}{\varepsilon C_j+\varepsilon^2D_j- d}d\tau
\notag\\&
 +\gamma_{3-j}w\Big(1+\varepsilon  b_jf'_j(w)\Big)\fint_{\mathbb{T}}\frac{\big(C_j+2\varepsilon D_j\big)\big(\overline{\tau}+\varepsilon b_{3-j} \overline{ f_{3-j}(\tau)}\big)\big(1+\varepsilon b_{3-j} f'_{3-j}(\tau)\big)}{\big(\varepsilon C_j+\varepsilon^2D_j- d\big)^2}d\tau \Bigg\}.\notag
\end{align}
with
$$
A=\tau-w,\quad B_j=f_j(\tau)-f_j(w),\quad C_j=\big(b_{3-j}\tau+b_j w\big),\quad D_j= b_{3-j}^2 f_{3-j}(\tau)+b_j^2 f_j(w).
$$
For $\varepsilon=0$ we have
\begin{align}\label{vareps0fj}
\partial_\varepsilon F_j(0,g)=\textnormal{Im}\Bigg\{&2\Omega\Big((-1)^j{Z}-(j-1)d\Big)w b_j f'_j(w)
+ b_j {\gamma_j} w\fint_{\mathbb{T}}\frac{A\overline{B}_j-\overline{A}B_j}{A^2}\Big[f'_j(\tau)-\frac{B_j}{A}\Big]d\tau
\\&+\frac{\gamma_{3-j}b_j}{d}w  f'_j(w)+\frac{\gamma_{3-j}b_j}{d^2}w^2 \Bigg\}.\notag
\end{align}
Consequently, for all $h=(\alpha_1,\alpha_2,h_1,h_2)\in \mathbb{R}\times \mathbb{R}\times X$, one finds 
\begin{align}\label{ffeps10}
\partial^2_{\varepsilon g} F_j(0,g)h=&\alpha_1\textnormal{Im}\Bigg\{2\Big((-1)^j{Z}-(j-1)d\Big)w b_j f'_j(w)
 \Bigg\}
 +\alpha_2\textnormal{Im}\Bigg\{2\Omega(-1)^jw b_j f'_j(w)
 \Bigg\}.
 \\ &+\textnormal{Im}\Bigg\{2\Omega\Big((-1)^j{Z}-(j-1)d\Big)w b_j h'_j(w)+\frac{\gamma_{3-j}b_j}{d}w  h'_j(w)\notag
\\&+ b_j {\gamma_j} w\fint_{\mathbb{T}}\frac{2i\textnormal{Im}\{A(\overline{f_j(\tau)}-\overline{f_j(w)})\}}{A^2}\Big[h'_j(\tau)-\frac{h_j(\tau)-h_j(w)}{A}\Big]d\tau\notag
\\&+ b_j {\gamma_j} w\fint_{\mathbb{T}}\frac{2i\textnormal{Im}\{A(\overline{h_j(\tau)}-\overline{h_j(w)})\}}{A^2}\Big[f'_j(\tau)-\frac{f_j(\tau)-f_j(w)}{A}\Big]d\tau
\Bigg\}.\notag
\end{align}
Then, for $g=g_0\triangleq\Big(\frac{\gamma_1+\gamma_2}{2d^2},\frac{d\gamma_2}{\gamma_1+\gamma_2},0,0\Big)$ one gets
\begin{align}\label{ffeps1}
\partial^2_{\varepsilon g} F_j(0,g_0)h=0.
\end{align}
Plugging the  identities \eqref{pepsilon}, \eqref{g-1}, \eqref{dngfj} and  \eqref{ffeps1}  into \eqref{vareps2f} yields
\begin{align}\label{pareps2f}
g''(0)
 &=\bigg(0, 0,
\dfrac{b_1^2\gamma_2}{d^3\gamma_1} \overline{w}^2,
\dfrac{b_2^2\gamma_1}{d^3\gamma_2} \overline{w}^2
\bigg) \cdot
\end{align}
Next, we move to the third order derivative
\begin{align}\label{varepsf3}
\partial^3_{\varepsilon\varepsilon\varepsilon}g(0)&=-D_{g} F(0,g_0)^{-1}\Big({\partial^3_{\varepsilon\varepsilon\varepsilon}} F(0,g_0)+3\partial^3_{\varepsilon \varepsilon g} F(0,g_0)g'(0)+3\partial^3_{\varepsilon gg} F(0,g_0)\big[g'(0),g'(0)\big]
\\&+3\partial^2_{\varepsilon g} F(0,g_0)g''(0)+3\partial^2_{ gg} F(0,g_0)\big[g'(0),g''(0)\big]+\partial^2_{ggg} F(0,g_0)\big[g'(0),g'(0),g'(0)\big]\Big).\notag
\end{align}
 Owing to the identity \eqref{vareps0fj}, for all $h=(\alpha_1,\alpha_2,h_1,h_2),k=(\beta_1,\beta_2,k_1,k_2)\in \mathbb{R}\times \mathbb{R}\times X$, one has one gets
\begin{align}\label{varepsgg}
\partial^3_{\varepsilon g g} F_j(0,g)[h,k]&=\beta_1\textnormal{Im}\bigg\{2\Big((-1)^j{Z}-(j-1)d\Big)w b_j h'_j(w)\bigg\}\\ &+
\beta_2\textnormal{Im}\bigg\{2(-1)^j\Omega w b_j h'_j(w)\bigg\}+
\alpha_1\textnormal{Im}\bigg\{2\Big((-1)^j{Z}-(j-1)d\Big)w b_j k'_j(w)\bigg\}\notag\\ &+
\alpha_2\textnormal{Im}\bigg\{2(-1)^j\Omega w b_j k'_j(w)\bigg\}+(\alpha_1\beta_2+\alpha_2\beta_1)\textnormal{Im}\bigg\{2(-1)^jw b_j f'_j(w)\bigg\}\notag
\\ & +\textnormal{Im}\Bigg\{ b_j {\gamma_j} w\fint_{\mathbb{T}}\frac{2i\textnormal{Im}\{A(\overline{k_j(\tau)}-\overline{k_j(w)})\}}{A^2}\Big[h'_j(\tau)-\frac{h_j(\tau)-h_j(w)}{A}\Big]d\tau\notag
\\&+ b_j {\gamma_j} w\fint_{\mathbb{T}}\frac{2i\textnormal{Im}\{A(\overline{h_j(\tau)}-\overline{h_j(w)})\}}{A^2}\Big[k'_j(\tau)-\frac{k_j(\tau)-k_j(w)}{A}\Big]d\tau \Bigg\}.\notag
\end{align}
Thus, for $g=g_0\triangleq\Big(\frac{\gamma_1+\gamma_2}{2d^2},\frac{d\gamma_2}{\gamma_1+\gamma_2},0,0\Big)$, $h=g'(0)$ and $k=g'(0)$ we find

\begin{align}\label{varepsgg3}
\partial^3_{\varepsilon g g} F_j(0,g_0)[g'(0),g'(0)]
&=2b_j {\gamma_j}\textnormal{Im}\Bigg\{  -w\frac{b_j^2\gamma_{3-j}^2}{d^4\gamma_j^2}\fint_{\mathbb{T}}\Big(1-\frac{(\overline{\tau}-\overline{w})^2}{(\tau-w)^2}\Big)\Big(\overline{\tau}^2-\frac{\overline{\tau}-\overline{w}}{\tau-w}\Big)d\tau\Bigg\}
\\ &=2b_j {\gamma_j}\textnormal{Im}\Bigg\{  -\frac{b_j^2\gamma_{3-j}^2}{d^4\gamma_j^2} \Bigg\}\notag\\ &=0.\notag
\end{align}
Next, differentiating \eqref{varepsfj} with respect to $\varepsilon$ gives
\begin{align}\label{verepsepsf}
&\partial^2_{\varepsilon\varepsilon} F_j(\varepsilon,g)=\textnormal{Im}\Bigg\{4\Omega b_j^2\overline{f_j(w)}w\Big(1+\varepsilon b_j f'_j(w)\Big) +4b_j^2\Omega\big(\overline{w}+2\varepsilon  b_j\overline{f_j(w)}\big)w  f'_j(w)
\\&+ 2b_j^2 {\gamma_j} w f'_j(w)\fint_{\mathbb{T}}\frac{A\overline{B}_j-\overline{A}B_j}{A(A+\varepsilon b_j B_j)}\Big[f'_j(\tau)-\frac{B_j}{A}\Big]d\tau\notag
\notag\\&-2 b_j^2 {\gamma_j} w\Big(1+2\varepsilon b_j f'_j(w)\Big)\fint_{\mathbb{T}}\frac{B_j(A\overline{B}_j-\overline{A}B_j)}{A(A+\varepsilon b_j B_j)^2}\Big[f'_j(\tau)-\frac{B_j}{A}\Big]d\tau\notag
\notag\\&+2\varepsilon b_j^3 {\gamma_j} w\Big(1+\varepsilon b_j f'_j(w)\Big)\fint_{\mathbb{T}}\frac{B_j^2(A\overline{B}_j-\overline{A}B_j)}{A(A+\varepsilon b_j B_j)^3}\Big[f'_j(\tau)-\frac{B_j}{A}\Big]d\tau
\notag\\ &-2\gamma_{3-j}  b_j b_{3-j} w f'_j(w)\fint_{\mathbb{T}}\frac{\overline{ f_{3-j}(\tau)}\big(1+\varepsilon b_{3-j} f'_{3-j}(\tau)\big)+\big(\overline{\tau}+\varepsilon b_{3-j} \overline{ f_{3-j}(\tau)}\big)  f'_{3-j}(\tau)}{\varepsilon C_j+\varepsilon^2D_j- d}d\tau
\notag\\ &+2\gamma_{3-j}w  b_jf'_j(w)\fint_{\mathbb{T}}\frac{\big(C_j+2\varepsilon D_j\big)\big(\overline{\tau}+\varepsilon b_{3-j} \overline{ f_{3-j}(\tau)}\big)\big(1+\varepsilon b_{3-j} f'_{3-j}(\tau)\big)}{\big(\varepsilon C_j+\varepsilon^2D_j- d\big)^2}d\tau
\notag\\ &-2b_{3-j}^2\gamma_{3-j}w\Big(1+\varepsilon  b_jf'_j(w)\Big)\fint_{\mathbb{T}}\frac{   f'_{3-j}(\tau)\overline{ f_{3-j}(\tau)}}{\varepsilon C_j+\varepsilon^2D_j- d}d\tau
\notag\\ &+2\gamma_{3-j}b_{3-j} w\Big(1+\varepsilon  b_jf'_j(w)\Big)\fint_{\mathbb{T}}\frac{ \big(C_j+2\varepsilon D_j\big)\overline{ f_{3-j}(\tau)}\big(1+\varepsilon b_{3-j} f'_{3-j}(\tau)\big)}{(\varepsilon C_j+\varepsilon^2D_j- d)^2}d\tau
\notag\\ &+2\gamma_{3-j}b_{3-j} w\Big(1+\varepsilon  b_jf'_j(w)\Big)\fint_{\mathbb{T}}\frac{ \big(C_j+2\varepsilon D_j\big)\big(\overline{\tau}+\varepsilon b_{3-j} \overline{ f_{3-j}(\tau)}\big)  f'_{3-j}(\tau)}{(\varepsilon C_j+\varepsilon^2D_j- d)^2}d\tau
\notag\\& +2\gamma_{3-j}w\Big(1+\varepsilon  b_jf'_j(w)\Big)\fint_{\mathbb{T}}\frac{ D_j\big(\overline{\tau}+\varepsilon b_{3-j} \overline{ f_{3-j}(\tau)}\big)\big(1+\varepsilon b_{3-j} f'_{3-j}(\tau)\big)}{\big(\varepsilon C_j+\varepsilon^2D_j- d\big)^2}d\tau
\notag\\& -2\gamma_{3-j}w\Big(1+\varepsilon  b_jf'_j(w)\Big)\fint_{\mathbb{T}}\frac{\big(C_j+2\varepsilon D_j\big)^2\big(\overline{\tau}+\varepsilon b_{3-j} \overline{ f_{3-j}(\tau)}\big)\big(1+\varepsilon b_{3-j} f'_{3-j}(\tau)\big)}{\big(\varepsilon C_j+\varepsilon^2D_j- d\big)^3}d\tau \Bigg\},\notag
\end{align}
where we have used the notation 
$$
A=\tau-w,\quad B_j=f_j(\tau)-f_j(w),\quad C_j=\big(b_{3-j}\tau+b_j w\big),\quad D_j= b_{3-j}^2 f_{3-j}(\tau)+b_j^2 f_j(w).
$$
At $\varepsilon=0$ one has
\begin{align*}
\partial^2_{\varepsilon\varepsilon} F_j(0,g)&=2\textnormal{Im}\Bigg\{2\Omega b_j^2\Big(w \overline{f_j(w)}+ f'_j(w)\Big)
+ b_j^2 {\gamma_j} w f'_j(w)\fint_{\mathbb{T}}\frac{A\overline{B}_j-\overline{A}B_j}{A^2}\Big[f'_j(\tau)-\frac{B_j}{A}\Big]d\tau\notag
\\&- b_j^2 {\gamma_j} w\fint_{\mathbb{T}}\frac{B_j(A\overline{B}_j-\overline{A}B_j)}{A^3}\Big[f'_j(\tau)-\frac{B_j}{A}\Big]d\tau\notag
+\frac{1}{d^2}\gamma_{3-j}w^2  b_j^2f'_j(w)
\\ &+\frac{1}{d}b_{3-j}^2\gamma_{3-j}w\fint_{\mathbb{T}}{   f'_{3-j}(\tau)\overline{ f_{3-j}(\tau)}}d\tau
 +\frac{1}{d^2}\gamma_{3-j}b_j^2wf_j(w)
 +\frac{1}{d^3}\gamma_{3-j}b_j^2w^3 \Bigg\}.
\end{align*}
Thus, for all $h=(\alpha_1,\alpha_2,h_1,h_2)\in \mathbb{R}\times \mathbb{R}\times X$ one finds
\begin{align}\label{epsepsg}
\partial^3_{\varepsilon\varepsilon g} F_j(0,g)h&=2\alpha_1\textnormal{Im}\Bigg\{2 b_j^2\Big(w\overline{f_j(w)}+ f'_j(w)\Big)
\Bigg\}\\ & +2\textnormal{Im}\Bigg\{2\Omega b_j^2\Big(w\overline{h_j(w)}+  h'_j(w)\Big)+\frac{1}{d^2}\gamma_{3-j}w^2  b_j^2h'_j(w)
 +\frac{1}{d^2}\gamma_{3-j}b_j^2wh_j(w)\notag
\\&+ b_j^2 {\gamma_j} w \fint_{\mathbb{T}}\Big[h'_j(w)-\frac{h_j(\tau)-h_j(w)}{A}\Big]\frac{A\overline{B}_j-\overline{A}B_j}{A^2}\Big[f'_j(\tau)-\frac{B_j}{A}\Big]d\tau\notag
\\&+ b_j^2 {\gamma_j} w \fint_{\mathbb{T}}\Big[f'_j(w)-\frac{B_j}{A}\Big]\frac{2i\textnormal{Im}\{A(\overline{h_j(\tau)}-\overline{h_j(w)})\}}{A^2}\Big[f'_j(\tau)-\frac{B_j}{A}\Big]d\tau\notag
\\&+ b_j^2 {\gamma_j} w \fint_{\mathbb{T}}\Big[f'_j(w)-\frac{B_j}{A}\Big]\frac{A\overline{B}_j-\overline{A}B_j}{A^2}\Big[h'_j(\tau)-\frac{h(\tau)-h_j(w)}{A}\Big]d\tau\notag
 \Bigg\}\\ &+\frac{2}{d}b_{3-j}^2\gamma_{3-j}\textnormal{Im}\Bigg\{w\fint_{\mathbb{T}}{   f'_{3-j}(\tau)\overline{ h_{3-j}(\tau)}}d\tau+w\fint_{\mathbb{T}}{   h'_{3-j}(\tau)\overline{ f_{3-j}(\tau)}}d\tau
 \Bigg\}.\notag
\end{align}
Therefore, for $g=g_0\triangleq\Big(\frac{\gamma_1+\gamma_2}{2d^2},\frac{d\gamma_2}{\gamma_1+\gamma_2},0,0\Big)$, $h=g'(0)$ we get
\begin{align}\label{varepsepsg1derv}
\partial^3_{\varepsilon\varepsilon g} F_j(0,g_0)g'(0)&=\frac{4b_j^3}{d^4}\frac{\gamma_{3-j}}{\gamma_j}(\gamma_1+\gamma_2)\textnormal{Im}\Big\{w^2
 \Big\}.
\end{align}
Next, in view of \eqref{dngfj0}  one has
\begin{equation}\label{dngfjn2}
\partial^3_{gg} F(0,g_0)\big[g'(0),g''(0)\big]=0
\end{equation}
and 
\begin{equation}\label{dngfjn}
\partial^3_{ggg} F(0,g_0)\big[h,k,l\big]=0\quad \textnormal{for all }\quad h,k,l\in  \mathbb{R}\times \mathbb{R}\times X.
\end{equation}
Putting together the identities \eqref{pepsilon}, \eqref{varepsgg3}, \eqref{varepsepsg1derv}, \eqref{dngfjn2} \eqref{dngfjn}  and \eqref{varepsf3} we conclude that
\begin{align}\label{pareps2ff}
\partial^3_{\varepsilon\varepsilon\varepsilon} g(0,w)
 &=\frac{12}{d^4}\bigg(0, 0,
b_1^3\Big(1+\dfrac{\gamma_2}{\gamma_1}\Big)\dfrac{\gamma_2}{\gamma_1}\overline{w}+b_1^3\dfrac{\gamma_2}{\gamma_1}\overline{w}^3,
b_2^3\Big(1+\dfrac{\gamma_1}{\gamma_2}\Big)\dfrac{\gamma_1}{\gamma_2}\overline{w}+b_2^3\dfrac{\gamma_1}{\gamma_2}\overline{w}^3 
\bigg)
\cdot
\end{align}
Now we move to the fourth order derivative in $\varepsilon$ of $g$. By the composition rule we get
\begin{align*}
\partial^4_{\varepsilon\varepsilon\varepsilon\varepsilon}g(0)=&-D_{g} F(0,g_0)^{-1}\Big({\partial^4_{\varepsilon\varepsilon\varepsilon\varepsilon}} F(0,g_0)+4\partial^4_{\varepsilon\varepsilon \varepsilon g} F(0,g_0)g'(0)
+6\partial^3_{\varepsilon\varepsilon g} F(0,g_0)g^{(2)}(0)\\ &+4\partial_{\varepsilon g} F(0,g_0)g^{(3)}(0) +6\partial^4_{\varepsilon \varepsilon g g} F(0,g_0)\big[g'(0),g'(0)\big]+12\partial^3_{\varepsilon g g} F(0,g_0)\big[g''(0),g'(0)\big]\notag\\ &+4\partial^4_{\varepsilon g g g} F(0,g_0)\big[g'(0),g'(0),g'(0)\big]\notag
+3\partial^2_{ gg} F(0,g_0)\big[g''(0),g''(0)\big]\\&+4\partial^2_{ gg} F(0,g_0)\big[g^{(3)}(0),g'(0)\big]+6\partial^3_{ ggg} F(0,g_0)\big[g'(0),g''(0),g'(0)\big]\notag\\ &+\partial^4_{gggg} F(0,g_0)\big[g'(0),g'(0),g'(0),g'(0)\big]\Big).\notag
\end{align*}
In view of \eqref{dngfj0} one has
\begin{equation}\label{pargg0}
\partial^2_{ gg} F(0,g_0)\big[g^{(3)}(0),g'(0)\big]=\partial^2_{ gg} F(0,g_0)\big[g''(0),g''(0)\big]=0
\end{equation}
and 
\begin{equation}\label{parggg0}
\partial^4_{gggg} F(0,g_0)\big[g'(0),g'(0),g'(0),g'(0)\big]=\partial^3_{ ggg} F(0,g_0)\big[g'(0),g''(0),g'(0)\big]=0.
\end{equation}
Moreover, from \eqref{ffeps1} we find
\begin{equation}\label{parggg0}
\partial_{\varepsilon g} F(0,g_0)g^{(3)}(0) =0
\end{equation}
and by \eqref{varepsgg} we obtain
\begin{align*}
\partial^3_{\varepsilon g g g} F_j(0,g)[h,k,l]&=(\beta_1\sigma_2+\beta_2\sigma_1)\textnormal{Im}\Bigg\{2(-1)^jw b_j h'_j(w)\Bigg\}+(\alpha_1\sigma_2+\alpha_2\sigma_1)\textnormal{Im}\Bigg\{2(-1)^jw b_j k'_j(w)\Bigg\}\\ &+(\beta_1\alpha_2+\beta_2\alpha_1)\textnormal{Im}\Bigg\{2(-1)^jw b_j l'_j(w)\Bigg\}.
\end{align*}
Consequently 
\begin{align*}
\partial^3_{\varepsilon g g g} F_j(0,g_0)[g'(0),g'(0),g'(0)]&=0,\quad \textnormal{for all }\quad h,k,l\in  \mathbb{R}\times \mathbb{R}\times X.
\end{align*}
Therefore, we conclude that
\begin{align}\label{varepsf41}
g^{(4)}(0)=-D_{g} F(0,g_0)^{-1}&\Big({\partial^4_{\varepsilon\varepsilon\varepsilon\varepsilon}} F(0,g_0)+4\partial^4_{\varepsilon\varepsilon \varepsilon g} F(0,g_0)g'(0)+6\partial^3_{\varepsilon\varepsilon g} F(0,g_0)g^{(2)}(0)\\ & +6\partial^4_{\varepsilon \varepsilon g g} F(0,g_0)\big[g'(0),g'(0)\big]+12\partial^3_{\varepsilon g g} F(0,g_0)\big[g''(0),g'(0)\big]\Big).\notag
\end{align}
From \eqref{verepsepsf}  we get
\begin{align*}
\partial^4_{\varepsilon\varepsilon\varepsilon g}  F_j(0,g_0)h=\frac{6\gamma_{3-j}}{d^3}\textnormal{Im}\Bigg\{&w^3  b_j^3h'_j(w)
  +b_{3-j}w\fint_{\mathbb{T}}{(b_jw+b_{3-j}\tau)^2\big[\overline{ h_{3-j}(\tau)}+\overline{\tau} h'_{3-j}(\tau)\big]}d\tau\\ &+2w\fint_{\mathbb{T}}\big(b_{j}^2{ h_{j}(w)}+{b_{3-j}^2{ h_{3-j}(\tau)}\big)(b_jw+b_{3-j}\tau)\overline{\tau}}d\tau \Bigg\},
\end{align*}
for all $h\in  \mathbb{R}\times \mathbb{R}\times X$.
Replacing $h$ by $g'(0)=\Big(0,0,\frac{b_1\gamma_2}{d^2\gamma_1}\overline{w},\frac{b_2\gamma_1}{d^2\gamma_2}\overline{w}\Big)$ gives
\begin{align*}
\partial^4_{\varepsilon\varepsilon\varepsilon g} F_j(0,g_0)g'(0)=\frac{6\gamma_{3-j}}{d^5}\textnormal{Im}\Bigg\{&-\frac{b_j^4\gamma_{3-j}}{\gamma_j}w+\frac{b_{3-j}^2\gamma_j}{\gamma_{3-j}}w\fint_{\mathbb{T}}{(b_jw+b_{3-j}\tau)^2\big[\tau-\overline{\tau}^3\big]}d\tau \\ &+2w\fint_{\mathbb{T}}\Big(\frac{b_{j}^3\gamma_{3-j}}{\gamma_{j}}\overline{w}+\frac{b_{3-j}^3\gamma_j}{\gamma_{3-j}}\overline{\tau}\Big)(b_jw+b_{3-j}\tau)\overline{\tau}d\tau \Bigg\}.
\end{align*}
Thus
\begin{align}\label{parepsepsepsg}
\partial^4_{\varepsilon\varepsilon\varepsilon g} F_j(0,g_0)g'(0)&=\frac{6}{d^5}\bigg(\gamma_j b_{3-j}^4+\frac{\gamma_{3-j}^2}{\gamma_j}b_{j}^4\bigg)\textnormal{Im}\Big\{{w}\Big\}.
\end{align}
Replacing $h$ by $g''(0)=\Big(0,0,\frac{b_1^2\gamma_2}{d^3\gamma_1}\overline{w}^2,\frac{b_2^2\gamma_1}{d^3\gamma_2}\overline{w}^2\Big)$ in  \eqref{epsepsg} we get
\begin{align}\label{parepsepsg2}
\partial^3_{\varepsilon\varepsilon g} F_j(0,g_0)g''(0)&=\frac{2b_j^4}{d^5}\frac{\gamma_{3-j}}{\gamma_j}\textnormal{Im}\Bigg\{(\gamma_1+\gamma_2)\Big(w^3- 2\overline{w}^3\Big)-\gamma_{3-j}
 \overline{w}
 \Bigg\}\\ &=\frac{6b_j^4}{d^5}\frac{\gamma_{3-j}}{\gamma_j}(\gamma_1+\gamma_2)\textnormal{Im}\Big\{w^3\Big\}+\frac{2b_j^4}{d^5}\frac{\gamma_{3-j}^2}{\gamma_j}\textnormal{Im}\Big\{
{w}
 \Big\}.\notag
\end{align}
From \eqref{varepsgg} we get
\begin{align*}
&\partial^3_{\varepsilon g g} F_j(0,g_0)[g''(0),g'(0)]\\ &=b_j {\gamma_j} \textnormal{Im}\Bigg\{ w\fint_{\mathbb{T}}\frac{2i\textnormal{Im}\{A(\overline{\partial_\varepsilon f_j(0,\tau)}-\overline{\partial_\varepsilon f_j(0,w)})\}}{A^2}\Big[\partial^3_{\varepsilon \varepsilon \tau}f_j(0,\tau)-\frac{\partial^2_{\varepsilon \varepsilon}f_j(0,\tau)-\partial^2_{\varepsilon \varepsilon}f_j(0,w)}{A}\Big]d\tau
\\&+  w\fint_{\mathbb{T}}\frac{2i\textnormal{Im}\{A(\overline{\partial^2_{\varepsilon \varepsilon}f_j(0,\tau)}-\overline{\partial^2_{\varepsilon \varepsilon}f_j(0,w)})\}}{A^2}\Big[\partial^2_{\varepsilon\tau} f_j(0,\tau)-\frac{\partial_\varepsilon f_j(0,\tau)-\partial_\varepsilon f_j(0,w)}{A}\Big]d\tau \Bigg\}
\\ &=\textnormal{Im}\Bigg\{- b_j {\gamma_j} \frac{b_j^3\gamma_{3-j}^2}{d^5\gamma_j^2} w\fint_{\mathbb{T}}\bigg(1- \frac{(\overline{\tau}-\overline{w})^2}{(\tau-w)^2}\bigg)\bigg(2\overline{\tau}^3+\frac{\overline{\tau}^2-\overline{w}^2}{\tau-w}\bigg)d\tau
\\&- b_j {\gamma_j} \frac{b_j^3\gamma_{3-j}^2}{d^5\gamma_j^2}  w\fint_{\mathbb{T}}\bigg(w+\tau- \frac{(\overline{\tau}-\overline{w})(\overline{\tau}^2-\overline{w}^2)}{(\tau-w)^2}\bigg)\Big(\overline{\tau}^2+\frac{\overline{\tau}-\overline{w}}{\tau-w}\Big)d\tau \Bigg\}.
\end{align*}
Now using the fact that
\begin{equation*}
\overline{\tau}-\overline{w}=-(\tau-w)\overline{w}\overline{\tau}\quad\forall w,\tau\in \mathbb{T},
\end{equation*} 
we get
\begin{align}\label{parepsggg2g1}
&\partial^3_{\varepsilon g g} F_j(0,g_0)[g''(0),g'(0)]=-\frac{b_j^4\gamma_{3-j}^2}{d^5\gamma_j} \textnormal{Im}\Big\{{w}\Big\}.
\end{align}
%
%
From \eqref{epsepsg} we find
%
\begin{align*}
\partial^3_{\varepsilon\varepsilon g g} F_j(0,g_0)[h,k]&=\alpha_1\textnormal{Im}\Bigg\{4 b_j^2\Big(w\overline{k_j(w)}+ k'_j(w)\Big)
\Bigg\} +2\beta_1\textnormal{Im}\Bigg\{2 b_j^2\Big(w\overline{h_j(w)}+  h'_j(w)\Big)\notag \Bigg\}
\\ &+\frac{2}{d}b_{3-j}^2\gamma_{3-j}\textnormal{Im}\Bigg\{w\fint_{\mathbb{T}}{   k'_{3-j}(\tau)\overline{ h_{3-j}(\tau)}}d\tau+w\fint_{\mathbb{T}}{   h'_{3-j}(\tau)\overline{ k_{3-j}(\tau)}}d\tau
 \Bigg\}.
\end{align*}
Therefore 
\begin{align}\label{epsepsggg1g1}
\partial^3_{\varepsilon\varepsilon g g} F_j(0,g_0)[g'(0),g'(0)]&=-\frac{4b_{3-j}^4}{d^5}\frac{\gamma_j^2}{\gamma_{3-j}}\textnormal{Im}\Big\{w \Big\}.
\end{align}
Plugging in the identities \eqref{pepsilon}, \eqref{parepsepsepsg} \eqref{parepsepsg2}, \eqref{parepsggg2g1} and \eqref{epsepsggg1g1} into \eqref{varepsf41} we get
\begin{align}\label{varepsf41f}
g^{(4)}(0)=-\frac{24}{d^{5}}D_{g} F(0,g_0)^{-1}
\Bigg(&\begin{pmatrix}
\gamma_2 b_1^4\\
\gamma_1 b_2^4 
\end{pmatrix}\textnormal{Im}\big\{ {w}^{5}\big\}+\frac{6}{4}(\gamma_1+\gamma_2)\begin{pmatrix}
\frac{\gamma_2}{\gamma_1} b_1^4\\
\frac{\gamma_1}{\gamma_2} b_2^4 
\end{pmatrix}\textnormal{Im}\big\{ {w}^3\big\}\\ & +\begin{pmatrix}
\gamma_1 b_2^4-\frac{\gamma_1^2}{\gamma_2} b_2^4+\frac{\gamma_2^2}{\gamma_1} b_1^4\\
\gamma_2 b_1^4 -\frac{\gamma_2^2}{\gamma_1} b_1^4+\frac{\gamma_1^2}{\gamma_2} b_2^4 
\end{pmatrix}\textnormal{Im}\big\{ {w}\big\}.\notag\Bigg)
\end{align}
In view of \eqref{g-1} we conclude that
\begin{align}
g^{(4)}(0)= \frac{24}{d^{5}}\Bigg(&
\frac{{\gamma_1} b_2^4+{\gamma_2} b_1^4 }{2d},
 \frac{d^{2}\big(\gamma_2^4b_1^4-\gamma_1^4b_2^4\big)}{\gamma_1\gamma_2(\gamma_1+\gamma_2)^2},
\frac{1}{4}\frac{\gamma_2}{\gamma_1} b_1^4\overline{w}^4+\frac{3}{4}\frac{\gamma_2}{\gamma_1}\big(\frac{\gamma_2}{\gamma_1}+1)b_1^4\overline{w}^2,\\ &
\frac{1}{4}\frac{\gamma_1}{\gamma_2} b_2^4\overline{w}^4+\frac{3}{4}\frac{\gamma_1}{\gamma_2}\big(\frac{\gamma_1}{\gamma_2}+1)b_2^4\overline{w}^2
\Bigg).\notag
\end{align}

 \end{proof}
\subsection{Counter-rotating pairs}  
Next we shall prove the following expansion.
\begin{proposition}\label{asymf00}
The  conformal mappings of the counter-rotating domains, $\phi_j:\mathbb{T}\to \partial D_j^\varepsilon$,   have the expansions 

\begin{align*}
\phi_j(\varepsilon,w)&=w-\Big(\frac{\varepsilon b_j }{ d}\Big)^2\overline{w}-{\frac12}\Big(\frac{\varepsilon b_j }{ d}\Big)^3\overline{w}^2-{\frac13}\Big(\frac{\varepsilon b_j }{ d}\Big)^4\overline{w}^3-{\frac14}\Big(\frac{\varepsilon b_j }{ d}\Big)^5\overline{w}^4+o(\varepsilon^5),
\end{align*}
for all $w\in\mathbb{T}$. Moreover,  
$$
U(\varepsilon)=\frac{\gamma_1}{2d}\Big(1+\frac{\varepsilon^4}{d^4}(2b_1^4+b_2^4)\Big)+o(\varepsilon^4)
$$
$$
\gamma_2(\varepsilon)={\gamma_1}\Big(1+\frac{\varepsilon^4}{d^4}(b_1^4-b_2^4)\Big)+o(\varepsilon^4).
$$
\end{proposition}
\begin{proof} 
Recall that
\begin{equation*}
\phi_j(\varepsilon,w)=w+\varepsilon b_j f_j(\varepsilon,w), \quad j=1,2.
\end{equation*}
Moreover, in view of Proposition \ref{propasym} one has
\begin{equation}\label{gf1f2counter}
F\Big(\varepsilon, g(\varepsilon)\Big)\triangleq  F\Big(\varepsilon, U(\varepsilon),\gamma_2(\varepsilon),f_1(\varepsilon), f_2(\varepsilon)\Big)=0
\end{equation}
and
\begin{equation}\label{f0}
g(0)=\Big(U(0),\gamma_2(0),f_1(0), f_2(0)\Big)=\Big(\frac{\gamma_1}{2d},\gamma_1,0,0\Big).
\end{equation}
where $F$ is defined by \eqref{fj2}.
By the composition rule we get
\begin{equation}\label{1order0counter}
\partial_\varepsilon g(0)  =-D_{g} F(0,g_0)^{-1}\partial_\varepsilon F(0,g_0)\cdot
\end{equation}
In view of \eqref{fj2} one has
\begin{align*}
{F}_j\big(\varepsilon,g_0\big)&= \gamma_1\,\textnormal{Im}\bigg\{ \Big(\frac{1}{d}+  \displaystyle\fint_{\mathbb{T}}\frac{\overline{\tau}}{\varepsilon \big(b_{3-j}\tau+b_j w\big)- d}d\tau\Big)w\bigg\}\\ &= \gamma_1\,\textnormal{Im}\bigg\{ \Big(\frac{1}{d}-\sum_{n=0}^\infty\frac{\varepsilon^{n}}{d^{n+1}} \displaystyle\fint_{\mathbb{T}}{\overline{\tau}} \big(b_{3-j}\tau+b_j w\big)^n d\tau \Big)w\bigg\} \\ &=-\frac{\gamma_1}{d}\sum_{n=1}^\infty\frac{\varepsilon^{n}b_j^n}{d^{n}}  \textnormal{Im}\big\{ w^{n+1}\big\}\cdot
\end{align*}
Thus
\begin{align}\label{pepsiloncounter}
\partial_\varepsilon^n F(0,g_0)(w)&=-\frac{n!\gamma_1}{d^{n+1}}\begin{pmatrix}
 b_1^n\\
 b_2^n 
\end{pmatrix}\textnormal{Im}\big\{ {w}^{n+1}\big\}\cdot
\end{align}
 Recall from Proposition  \ref{prop12} that
for all $h=(\alpha_1,\alpha_2,h_1,h_2)\in \mathbb{R}\times\mathbb{R}\times X$
with
$$
h_j(w)=\sum_{n\geq 1}a_n^j\overline{w}^{n},\quad j=1,2,
$$
 one has
\begin{align}
D_{g}F &(0,g_0)h(w)={\alpha_1}\begin{pmatrix}
2  \\
2
\end{pmatrix}\textnormal{Im}\big\{{w}\big\}-\dfrac{\alpha_2}{d}\begin{pmatrix}
 1 \\
0
\end{pmatrix}\textnormal{Im}\big\{{w}\big\}-\gamma_1\sum_{n\geq 1}n\begin{pmatrix}
a_n^1\\
a_n^2
\end{pmatrix}\textnormal{Im}\big\{{w}^{n+1}\big\}.
\end{align}
Then, for any $k\in {Y}$ with the expansion
$$
k(w)=\sum_{n\geq 0}\begin{pmatrix}
A_n \\
B_n 
\end{pmatrix}\textnormal{Im}\{{w}^{n+1}\},
$$
one has
\begin{align}\label{g-1counter}
D_{g}F(0,g_0)^{-1}k(w)&=-\bigg(-
\frac{B_0}{2},
d(A_0-B_0),
\displaystyle\sum_{n\geq 1}\dfrac{A_n}{\gamma_1n}\overline{w}^{n},
\displaystyle\sum_{n\geq 1}\dfrac{B_n}{\gamma_1n}  \overline{w}^{n}
\bigg)\cdot
\end{align}
Combining the last identity with \eqref{pepsiloncounter}, \eqref{1order0counter} and \eqref{g-1counter} we get
\begin{align}\label{parepsfcounter}
\partial_\varepsilon g(0,w) &=-\frac{1}{d^2}\bigg(
0,0,
b_1\overline{w},
b_2 \overline{w}
\bigg)\cdot
\end{align}
Let's move to the second order derivative in $\varepsilon$  at $0$. By the composition rule we get 
\begin{align}\label{vareps2fcounter}
\partial^2_{\varepsilon\varepsilon} g(0) &=-D_{g} F(0,g_0)^{-1}\Big(\partial^2_{\varepsilon} F(0,g_0)+2D^2_{\varepsilon g} F(0,g_0)g'(0)+D^2_{gg} F(0,g_0)[g'(0),g'(0)]\Big).
\end{align}
Observe  that
\begin{align}\label{dngfjncounter}
{F}_j(0,g)&=\textnormal{Im}\bigg\{2U w-\gamma_jf'_j(w)-\frac{\gamma_{3-j}}{d} w\bigg\}.
\end{align}
Thus, for all $h=(\alpha_1,\alpha_2,h_1,h_2), k=(\beta_1,\beta_2,k_1,k_2)\in \mathbb{R}\times \mathbb{R}\times X$, one has
\begin{equation}\label{dngfj0counter}
\partial^2_{gg} F_1(0,g)\big[h,k\big]=0,\quad \partial^2_{gg} F_2(0,g)\big[h,k\big]=-\alpha_2\textnormal{Im}\bigg\{k'_2(w)\bigg\}-\beta_2\textnormal{Im}\bigg\{h'_2(w)\bigg\}\cdot
\end{equation}
From the identity \eqref{parepsfcounter} we get
\begin{equation}\label{dngfjcounter}
\partial^2_{gg} F_j(0,g_0)\big[g'(0),g'(0)\big]=0\cdot
\end{equation}
Next, differentiating the identity \eqref{fj2} with respect to $\varepsilon$ we get
\begin{align}\label{varepsfjcounter}
&\partial_\varepsilon  F_j(\varepsilon,g)=\textnormal{Im}\Bigg\{2Uw b_j f'_j(w)
+ b_j {\gamma_j} w\Big(1+2\varepsilon b_j f'_j(w)\Big)\fint_{\mathbb{T}}\frac{A\overline{B}_j-\overline{A}B_j}{A(A+\varepsilon b_j B_j)}\Big[f'_j(\tau)-\frac{B_j}{A}\Big]d\tau
\\&-\varepsilon b_j^2 {\gamma_j} w\Big(1+\varepsilon b_j f'_j(w)\Big)\fint_{\mathbb{T}}\frac{B_j(A\overline{B}_j-\overline{A}B_j)}{A(A+\varepsilon b_j B_j)^2}\Big[f'_j(\tau)-\frac{B_j}{A}\Big]d\tau
\notag\\ &+\gamma_{3-j}b_j w  f'_j(w)\fint_{\mathbb{T}}\frac{\big(\overline{\tau}+\varepsilon b_{3-j} \overline{ f_{3-j}(\tau)}\big)\big(1+\varepsilon b_{3-j} f'_{3-j}(\tau)\big)}{\varepsilon C_j+\varepsilon^2D_j- d}d\tau 
\notag\\ &
+\gamma_{3-j} b_{3-j} w\Big(1+\varepsilon  b_jf'_j(w)\Big)\fint_{\mathbb{T}}\frac{\overline{ f_{3-j}(\tau)}\big(1+\varepsilon b_{3-j} f'_{3-j}(\tau)\big)+\big(\overline{\tau}+\varepsilon b_{3-j} \overline{ f_{3-j}(\tau)}\big) f'_{3-j}(\tau)}{\varepsilon C_j+\varepsilon^2D_j- d}d\tau
\notag\\&
 -\gamma_{3-j}w\Big(1+\varepsilon  b_jf'_j(w)\Big)\fint_{\mathbb{T}}\frac{\big(C_j+2\varepsilon D_j\big)\big(\overline{\tau}+\varepsilon b_{3-j} \overline{ f_{3-j}(\tau)}\big)\big(1+\varepsilon b_{3-j} f'_{3-j}(\tau)\big)}{\big(\varepsilon C_j+\varepsilon^2D_j- d\big)^2}d\tau \Bigg\},\notag
\end{align}
where we have used the notation 
$$
A=\tau-w,\quad B_j=f_j(\tau)-f_j(w),\quad C_j=\big(b_{3-j}\tau+b_j w\big),\quad D_j= b_{3-j}^2 f_{3-j}(\tau)+b_j^2 f_j(w).
$$
Consequently, for all $h=(\alpha_1,\alpha_2,h_1,h_2)\in \mathbb{R}\times \mathbb{R}\times X$, one finds 
\begin{align}\label{ffeps10counter}
\partial^2_{\varepsilon g} F_1(0,g)h=&\alpha_1\textnormal{Im}\Big\{2w b_j f'_j(w)\Big\}-\frac{\alpha_2b_j\delta_{1j}}{d}\textnormal{Im}\Bigg\{w  f'_j(w)-\frac{1}{d}w^2 \Bigg\}\\ &+\alpha_2\delta_{2j}\textnormal{Im}\Bigg\{ b_j  w\fint_{\mathbb{T}}\frac{A\overline{B}_j-\overline{A}B_j}{A^2}\Big[f'_j(\tau)-\frac{B_j}{A}\Big]d\tau \Bigg\}\notag \\ &+\textnormal{Im}\Bigg\{2Uw b_j h'_j(w)-\frac{\gamma_{3-j}b_j}{d}w  h'_j(w)\notag
\\&+ b_j {\gamma_j} w\fint_{\mathbb{T}}\frac{2i\textnormal{Im}\{A(\overline{f_j(\tau)}-\overline{f_j(w)})\}}{A^2}\Big[h'_j(\tau)-\frac{h_j(\tau)-h_j(w)}{A}\Big]d\tau\notag
\\&+ b_j {\gamma_j} w\fint_{\mathbb{T}}\frac{2i\textnormal{Im}\{A(\overline{h_j(\tau)}-\overline{h_j(w)})\}}{A^2}\Big[f'_j(\tau)-\frac{f_j(\tau)-f_j(w)}{A}\Big]d\tau
\Bigg\},\notag
\end{align}
where $\delta_{ij}$ is the Kronecker delta. 
Then, for $g=g_0\triangleq\Big(\frac{\gamma_1}{2d},\gamma_1,0,0\Big)$ and $h=g'(0)$ one gets
\begin{align}\label{ffeps1counter}
\partial^2_{\varepsilon g} F_j(0,g_0)g'(0)=0.
\end{align}
Plugging the  identities \eqref{pepsiloncounter}, \eqref{g-1counter}, \eqref{dngfjcounter} and  \eqref{ffeps1counter}  into \eqref{vareps2fcounter} yields
\begin{align}\label{pareps2fcounter}
g''(0)
 &=-\frac{1}{d^3}\bigg(0, 0,
{b_1^2} \overline{w}^2,
{b_2^2} \overline{w}^2
\bigg) \cdot
\end{align}
Next, we move to the third order derivative
\begin{align}\label{varepsf3counter}
\partial^3_{\varepsilon\varepsilon\varepsilon}g(0)&=-D_{g} F(0,g_0)^{-1}\Big({\partial^3_{\varepsilon\varepsilon\varepsilon}} F(0,g_0)+3\partial^3_{\varepsilon \varepsilon g} F(0,g_0)g'(0)+3\partial^3_{\varepsilon gg} F(0,g_0)\big[g'(0),g'(0)\big]
\\&+3\partial^2_{\varepsilon g} F(0,g_0)g''(0)+3\partial^2_{ gg} F(0,g_0)\big[g'(0),g''(0)\big]+\partial^2_{ggg} F(0,g_0)\big[g'(0),g'(0),g'(0)\big]\Big).\notag
\end{align}
By virtue of  the identity \eqref{ffeps10counter}, for all $h=(\alpha_1,\alpha_2,h_1,h_2),k=(\beta_1,\beta_2,k_1,k_2)\in \mathbb{R}\times \mathbb{R}\times X$, one has 
\begin{align}\label{varepsggcounter}
\partial^3_{\varepsilon g g} F_j(0,g)&[h,k]=\alpha_1\textnormal{Im}\Big\{2w b_j k'_j(w)\Big\}+\beta_1\textnormal{Im}\Big\{2w b_j h'_j(w)\Big\}\\ &-\frac{\alpha_{2}b_j\delta_{1j}}{d}\textnormal{Im}\Big\{w  k'_j(w) \Big\}-\frac{\beta_{2}b_j\delta_{1j}}{d}\textnormal{Im}\Big\{w  h'_j(w) \Big\} \notag
\\&+(\alpha_2+\beta_2)\delta_{2j}b_j\textnormal{Im}\Bigg\{ w\fint_{\mathbb{T}}\frac{2i\textnormal{Im}\{A(\overline{f_j(\tau)}-\overline{f_j(w)})\}}{A^2}\Big[h'_j(\tau)-\frac{h_j(\tau)-h_j(w)}{A}\Big]d\tau\notag
\\&+  w\fint_{\mathbb{T}}\frac{2i\textnormal{Im}\{A(\overline{h_j(\tau)}-\overline{h_j(w)})\}}{A^2}\Big[f'_j(\tau)-\frac{f_j(\tau)-f_j(w)}{A}\Big]d\tau\Bigg\}\notag
\\ &+ b_j {\gamma_j}\textnormal{Im}\Bigg\{ w\fint_{\mathbb{T}}\frac{2i\textnormal{Im}\{A(\overline{k_j(\tau)}-\overline{k_j(w)})\}}{A^2}\Big[h'_j(\tau)-\frac{h_j(\tau)-h_j(w)}{A}\Big]d\tau\notag
\\&+  w\fint_{\mathbb{T}}\frac{2i\textnormal{Im}\{A(\overline{h_j(\tau)}-\overline{h_j(w)})\}}{A^2}\Big[k'_j(\tau)-\frac{k_j(\tau)-k_j(w)}{A}\Big]d\tau
\Bigg\}.\notag
\end{align}

Thus, for $g=g_0\triangleq \Big(\frac{\gamma_1}{2d},\gamma_1,0,0\Big)$, $h=g'(0)$ and $k=g'(0)$ we find

\begin{align}\label{varepsgg3counter}
&\partial^3_{\varepsilon g g} F_j(0,g_0)[g'(0),g'(0)]=
2\gamma_1b_j \textnormal{Im}\Bigg\{  -w\frac{b_j^2}{d^4}\fint_{\mathbb{T}}\Big(1-\frac{(\overline{\tau}-\overline{w})^2}{(\tau-w)^2}\Big)\Big(\overline{\tau}^2-\frac{\overline{\tau}-\overline{w}}{\tau-w}\Big)d\tau\Bigg\}=0.
\end{align}
Next, differentiating \eqref{varepsfjcounter} with respect to $\varepsilon$ gives
\begin{align}\label{verepsepsfcounter22}
&\partial^2_{\varepsilon\varepsilon} F_j(\varepsilon,g)=\textnormal{Im}\Bigg\{ 2b_j^2 {\gamma_j} w f'_j(w)\fint_{\mathbb{T}}\frac{A\overline{B}_j-\overline{A}B_j}{A(A+\varepsilon b_j B_j)}\Big[f'_j(\tau)-\frac{B_j}{A}\Big]d\tau
\\&-2 b_j^2 {\gamma_j} w\Big(1+2\varepsilon b_j f'_j(w)\Big)\fint_{\mathbb{T}}\frac{B_j(A\overline{B}_j-\overline{A}B_j)}{A(A+\varepsilon b_j B_j)^2}\Big[f'_j(\tau)-\frac{B_j}{A}\Big]d\tau\notag
\notag\\&+2\varepsilon b_j^3 {\gamma_j} w\Big(1+\varepsilon b_j f'_j(w)\Big)\fint_{\mathbb{T}}\frac{B_j^2(A\overline{B}_j-\overline{A}B_j)}{A(A+\varepsilon b_j B_j)^3}\Big[f'_j(\tau)-\frac{B_j}{A}\Big]d\tau
\notag\\ &+2\gamma_{3-j}  b_j b_{3-j} w f'_j(w)\fint_{\mathbb{T}}\frac{\overline{ f_{3-j}(\tau)}\big(1+\varepsilon b_{3-j} f'_{3-j}(\tau)\big)+\big(\overline{\tau}+\varepsilon b_{3-j} \overline{ f_{3-j}(\tau)}\big)  f'_{3-j}(\tau)}{\varepsilon C_j+\varepsilon^2D_j- d}d\tau
\notag\\ &-2\gamma_{3-j}w  b_jf'_j(w)\fint_{\mathbb{T}}\frac{\big(C_j+2\varepsilon D_j\big)\big(\overline{\tau}+\varepsilon b_{3-j} \overline{ f_{3-j}(\tau)}\big)\big(1+\varepsilon b_{3-j} f'_{3-j}(\tau)\big)}{\big(\varepsilon C_j+\varepsilon^2D_j- d\big)^2}d\tau
\notag\\ &+2b_{3-j}^2\gamma_{3-j}w\Big(1+\varepsilon  b_jf'_j(w)\Big)\fint_{\mathbb{T}}\frac{   f'_{3-j}(\tau)\overline{ f_{3-j}(\tau)}}{\varepsilon C_j+\varepsilon^2D_j- d}d\tau
\notag\\ &-2\gamma_{3-j}b_{3-j} w\Big(1+\varepsilon  b_jf'_j(w)\Big)\fint_{\mathbb{T}}\frac{ \big(C_j+2\varepsilon D_j\big)\overline{ f_{3-j}(\tau)}\big(1+\varepsilon b_{3-j} f'_{3-j}(\tau)\big)}{(\varepsilon C_j+\varepsilon^2D_j- d)^2}d\tau
\notag\\ &-2\gamma_{3-j}b_{3-j} w\Big(1+\varepsilon  b_jf'_j(w)\Big)\fint_{\mathbb{T}}\frac{ \big(C_j+2\varepsilon D_j\big)\big(\overline{\tau}+\varepsilon b_{3-j} \overline{ f_{3-j}(\tau)}\big)  f'_{3-j}(\tau)}{(\varepsilon C_j+\varepsilon^2D_j- d)^2}d\tau
\notag\\& -2\gamma_{3-j}w\Big(1+\varepsilon  b_jf'_j(w)\Big)\fint_{\mathbb{T}}\frac{ D_j\big(\overline{\tau}+\varepsilon b_{3-j} \overline{ f_{3-j}(\tau)}\big)\big(1+\varepsilon b_{3-j} f'_{3-j}(\tau)\big)}{\big(\varepsilon C_j+\varepsilon^2D_j- d\big)^2}d\tau
\notag\\& +2\gamma_{3-j}w\Big(1+\varepsilon  b_jf'_j(w)\Big)\fint_{\mathbb{T}}\frac{\big(C_j+2\varepsilon D_j\big)^2\big(\overline{\tau}+\varepsilon b_{3-j} \overline{ f_{3-j}(\tau)}\big)\big(1+\varepsilon b_{3-j} f'_{3-j}(\tau)\big)}{\big(\varepsilon C_j+\varepsilon^2D_j- d\big)^3}d\tau \Bigg\},\notag
\end{align}
with $$
A=\tau-w,\quad B_j=f_j(\tau)-f_j(w),\quad C_j=\big(b_{3-j}\tau+b_j w\big),\quad D_j= b_{3-j}^2 f_{3-j}(\tau)+b_j^2 f_j(w).
$$
For $\varepsilon=0$ one has
\begin{align*}
&\partial^2_{\varepsilon\varepsilon} F_j(0,g)=2\textnormal{Im}\Bigg\{ b_j^2 {\gamma_j} w \fint_{\mathbb{T}}\frac{A\overline{B}_j-\overline{A}B_j}{A^2}\Big[f'_j(\tau)-\frac{B_j}{A}\Big]\Big[f'_j(w)-\frac{B_j}{A}\Big]d\tau
\\ &-\frac{1}{d^2}\gamma_{3-j}w^2  b_j^2f'_j(w)
-\frac{1}{d}b_{3-j}^2\gamma_{3-j}w\fint_{\mathbb{T}}{   f'_{3-j}(\tau)\overline{ f_{3-j}(\tau)}}d\tau
 -\frac{1}{d^2}\gamma_{3-j}b_j^2wf_j(w)
 -\frac{1}{d^3}\gamma_{3-j}b_j^2w^3 \Bigg\}.
\end{align*}
Thus, for all $h=(\alpha_1,\alpha_2,h_1,h_2)\in \mathbb{R}\times \mathbb{R}\times X$ one finds
\begin{align}\label{epsepsgcounter}
\partial^3_{\varepsilon\varepsilon g} &F_j(0,g)h=2\alpha_2\delta_{2j}\textnormal{Im}\Bigg\{ b_j^2  w \fint_{\mathbb{T}}\frac{A\overline{B}_j-\overline{A}B_j}{A^2}\Big[f'_j(\tau)-\frac{B_j}{A}\Big]\Big[f'_j(\tau)-\frac{B_j}{A}\Big]d\tau\Bigg\}
\\&\notag
-2\alpha_{2}\delta_{1j}\textnormal{Im}\Bigg\{ \frac{1}{d^2}w^2  b_j^2f'_j(w)\notag
+\frac{1}{d}b_{3-j}^2w\fint_{\mathbb{T}}{   f'_{3-j}(\tau)\overline{ f_{3-j}(\tau)}}d\tau
 +\frac{1}{d^2}b_j^2wf_j(w)
 +\frac{1}{d^3}b_j^2w^3 \Bigg\}
 \notag\\ &+2\textnormal{Im}\Bigg\{-\frac{1}{d^2}\gamma_{3-j}w^2  b_j^2h'_j(w)
 -\frac{1}{d^2}\gamma_{3-j}b_j^2wh_j(w)\notag
\\&+ b_j^2 {\gamma_j} w \fint_{\mathbb{T}}\Big[h'_j(w)-\frac{h_j(\tau)-h_j(w)}{A}\Big]\frac{A\overline{B}_j-\overline{A}B_j}{A^2}\Big[f'_j(\tau)-\frac{B_j}{A}\Big]d\tau\notag
\\&+ b_j^2 {\gamma_j} w \fint_{\mathbb{T}}\Big[f'_j(w)-\frac{B_j}{A}\Big]\frac{2i\textnormal{Im}\{A(\overline{h_j(\tau)}-\overline{h_j(w)})\}}{A^2}\Big[f'_j(\tau)-\frac{B_j}{A}\Big]d\tau\notag
\\&+ b_j^2 {\gamma_j} w \fint_{\mathbb{T}}\Big[f'_j(w)-\frac{B_j}{A}\Big]\frac{A\overline{B}_j-\overline{A}B_j}{A^2}\Big[h'_j(\tau)-\frac{h(\tau)-h_j(w)}{A}\Big]d\tau\notag
 \Bigg\}\\ &-\frac{2}{d}b_{3-j}^2\gamma_{3-j}\textnormal{Im}\Bigg\{w\fint_{\mathbb{T}}{   f'_{3-j}(\tau)\overline{ h_{3-j}(\tau)}}d\tau+w\fint_{\mathbb{T}}{   h'_{3-j}(\tau)\overline{ f_{3-j}(\tau)}}d\tau
 \Bigg\}.\notag
\end{align}
Therefore, for $g=g_0\triangleq\Big(\frac{\gamma_1}{2d},\gamma_1,0,0\Big)$, $h=g'(0)$ we get
\begin{align}\label{varepsepsg1dervcounter}
\partial^3_{\varepsilon\varepsilon g} F_j(0,g_0)g'(0)&=0.
\end{align}
Next, in view of \eqref{dngfj0counter}  one has
\begin{equation}\label{dngfjn2counter}
\partial^3_{gg} F(0,g_0)\big[g'(0),g''(0)\big]=0
\end{equation}
and 
\begin{equation}\label{dngfjncounter}
\partial^3_{ggg} F(0,g_0)\big[h,k,l\big]=0\quad \textnormal{for all }\quad h,k,l\in  \mathbb{R}\times \mathbb{R}\times X.
\end{equation}
Putting together the identities \eqref{pepsiloncounter}, \eqref{varepsgg3counter}, \eqref{varepsepsg1dervcounter}, \eqref{dngfjn2counter} \eqref{dngfjncounter}  and \eqref{varepsf3counter} we conclude that
\begin{align}\label{pareps2ffcounter}
\partial^3_{\varepsilon\varepsilon\varepsilon} g(0,w)
 &=-\frac{2}{d^4}\bigg(0, 0,
b_1^3\overline{w}^3,
b_2^3\overline{w}^3 
\bigg)
\cdot
\end{align}
Now we move to the fourth order derivative in $\varepsilon$ of $g$. By the composition rule we get
\begin{align*}
\partial^4_{\varepsilon\varepsilon\varepsilon\varepsilon}g(0)=&-D_{g} F(0,g_0)^{-1}\Big({\partial^4_{\varepsilon\varepsilon\varepsilon\varepsilon}} F(0,g_0)+4\partial^4_{\varepsilon\varepsilon \varepsilon g} F(0,g_0)g'(0)
+6\partial^3_{\varepsilon\varepsilon g} F(0,g_0)g^{(2)}(0)\\ &+4\partial_{\varepsilon g} F(0,g_0)g^{(3)}(0) +6\partial^4_{\varepsilon \varepsilon g g} F(0,g_0)\big[g'(0),g'(0)\big]+12\partial^3_{\varepsilon g g} F(0,g_0)\big[g''(0),g'(0)\big]\notag\\ &+4\partial^4_{\varepsilon g g g} F(0,g_0)\big[g'(0),g'(0),g'(0)\big]\notag
+3\partial^2_{ gg} F(0,g_0)\big[g''(0),g''(0)\big]\\&+4\partial^2_{ gg} F(0,g_0)\big[g^{(3)}(0),g'(0)\big]+6\partial^3_{ ggg} F(0,g_0)\big[g'(0),g''(0),g'(0)\big]\notag\\ &+\partial^4_{gggg} F(0,g_0)\big[g'(0),g'(0),g'(0),g'(0)\big]\Big).\notag
\end{align*}
In view of \eqref{dngfj0counter} one has
\begin{equation}\label{pargg0counter}
\partial^2_{ gg} F(0,g_0)\big[g^{(3)}(0),g'(0)\big]=\partial^2_{ gg} F(0,g_0)\big[g''(0),g''(0)\big]=0
\end{equation}
and 
\begin{equation}\label{parggg0counter}
\partial^4_{gggg} F(0,g_0)\big[g'(0),g'(0),g'(0),g'(0)\big]=\partial^3_{ ggg} F(0,g_0)\big[g'(0),g''(0),g'(0)\big]=0.
\end{equation}
Moreover, from \eqref{ffeps10counter} we find
\begin{equation}\label{parggg0counter}
\partial_{\varepsilon g} F(0,g_0)g^{(3)}(0) =0
\end{equation}
and by \eqref{varepsggcounter} we obtain
\begin{align*}
\partial^4_{\varepsilon g g g} F_j(0,g)[h,k,l]=&(\alpha_2+\beta_2)\delta_{2j}\textnormal{Im}\Bigg\{ b_jw\fint_{\mathbb{T}}\frac{2i\textnormal{Im}\{A(\overline{l_j(\tau)}-\overline{l_j(w)})\}}{A^2}\Big[h'_j(\tau)-\frac{h_j(\tau)-h_j(w)}{A}\Big]d\tau
\\&+ b_j  w\fint_{\mathbb{T}}\frac{2i\textnormal{Im}\{A(\overline{h_j(\tau)}-\overline{h_j(w)})\}}{A^2}\Big[l'_j(\tau)-\frac{l_j(\tau)-l_j(w)}{A}\Big]d\tau\Bigg\}\notag
\\&+\sigma_j\textnormal{Im}\Bigg\{ b_j  w\fint_{\mathbb{T}}\frac{2i\textnormal{Im}\{A(\overline{k_j(\tau)}-\overline{k_j(w)})\}}{A^2}\Big[h'_j(\tau)-\frac{h_j(\tau)-h_j(w)}{A}\Big]d\tau\notag
\\&+ b_j  w\fint_{\mathbb{T}}\frac{2i\textnormal{Im}\{A(\overline{h_j(\tau)}-\overline{h_j(w)})\}}{A^2}\Big[k'_j(\tau)-\frac{k_j(\tau)-k_j(w)}{A}\Big]d\tau
\Bigg\}.\notag
\end{align*}
Consequently 
\begin{align*}
\partial^3_{\varepsilon g g g} F_j(0,g_0)[g'(0),g'(0),g'(0)]&=0,\quad \textnormal{for all }\quad h,k,l\in  \mathbb{R}\times \mathbb{R}\times X.
\end{align*}
Therefore, we conclude that
\begin{align}\label{varepsf41}
g^{(4)}(0)=-D_{g} F(0,g_0)^{-1}&\Big({\partial^4_{\varepsilon\varepsilon\varepsilon\varepsilon}} F(0,g_0)+4\partial^4_{\varepsilon\varepsilon \varepsilon g} F(0,g_0)g'(0)+6\partial^3_{\varepsilon\varepsilon g} F(0,g_0)g^{(2)}(0)\\ & +6\partial^4_{\varepsilon \varepsilon g g} F(0,g_0)\big[g'(0),g'(0)\big]+12\partial^3_{\varepsilon g g} F(0,g_0)\big[g''(0),g'(0)\big]\Big).\notag
\end{align}
From \eqref{verepsepsfcounter22}  we get
\begin{align*}
\partial^4_{\varepsilon\varepsilon\varepsilon g}  F_j(0,g_0)h= -\frac{6\gamma_1}{d^3}\textnormal{Im}\Bigg\{&w^3  b_j^3h'_j(w)
  +b_{3-j}w\fint_{\mathbb{T}}{(b_jw+b_{3-j}\tau)^2\big[\overline{ h_{3-j}(\tau)}+\overline{\tau} h'_{3-j}(\tau)\big]}d\tau\\ &+2w\fint_{\mathbb{T}}\big(b_{j}^2{ h_{j}(w)}+{b_{3-j}^2{ h_{3-j}(\tau)}\big)(b_jw+b_{3-j}\tau)\overline{\tau}}d\tau \Bigg\},
\end{align*}
for all $h\in  \mathbb{R}\times \mathbb{R}\times X$.
Replacing $h$ by $g'(0)=\Big(0,0,-\frac{b_1}{d^2}\overline{w},-\frac{b_2}{d^2}\overline{w}\Big)$ gives
\begin{align*}
\partial^4_{\varepsilon\varepsilon\varepsilon g} F_j(0,g_0)g'(0)=-\frac{6\gamma_1}{d^5}\textnormal{Im}\Bigg\{&b_j^4w-b_{3-j}^2w\fint_{\mathbb{T}}{(b_jw+b_{3-j}\tau)^2\big[\tau-\overline{\tau}^3\big]}d\tau \\ &-2w\fint_{\mathbb{T}}\Big(b_{j}^3\overline{w}+b_{3-j}^3\overline{\tau}\Big)(b_jw+b_{3-j}\tau)\overline{\tau}d\tau \Bigg\}.
\end{align*}
Thus
\begin{align}\label{parepsepsepsgcounter}
\partial^4_{\varepsilon\varepsilon\varepsilon g} F_j(0,g_0)g'(0)&=\frac{6\gamma_1}{d^5}\bigg(b_{j}^4+b_{3-j}^4\bigg)\textnormal{Im}\Big\{{w}\Big\}.
\end{align}
Replacing $h$ by $g''(0)=\Big(0,0,-\frac{b_1^2}{d^3}\overline{w}^2,-\frac{b_2^2}{d^3}\overline{w}^2\Big)$ in  \eqref{epsepsgcounter} we get
\begin{align}\label{parepsepsg2counter},
\partial^3_{\varepsilon\varepsilon g} F_j(0,g_0)g''(0)&=\frac{2\gamma_1b_j^4}{d^5}\textnormal{Im}\Big\{
{w}
 \Big\}.
\end{align}
From \eqref{varepsggcounter} we get
%
\begin{align*}
&\partial^3_{\varepsilon g g} F_j(0,g_0)[g''(0),g'(0)]\\ &=\gamma_1\textnormal{Im}\Bigg\{ b_j  w\fint_{\mathbb{T}}\frac{2i\textnormal{Im}\{A(\overline{\partial_\varepsilon f_j(0,\tau)}-\overline{\partial_\varepsilon f_j(0,w)})\}}{A^2}\Big[\partial^3_{\varepsilon \varepsilon \tau}f_j(0,\tau)-\frac{\partial^2_{\varepsilon \varepsilon}f_j(0,\tau)-\partial^2_{\varepsilon \varepsilon}f_j(0,w)}{A}\Big]d\tau
\\&+ b_j  w\fint_{\mathbb{T}}\frac{2i\textnormal{Im}\{A(\overline{\partial^2_{\varepsilon \varepsilon}f_j(0,\tau)}-\overline{\partial^2_{\varepsilon \varepsilon}f_j(0,w)})\}}{A^2}\Big[\partial^2_{\varepsilon\tau} f_j(0,\tau)-\frac{\partial_\varepsilon f_j(0,\tau)-\partial_\varepsilon f_j(0,w)}{A}\Big]d\tau \Bigg\}.
\\ &=\gamma_1\,\textnormal{Im}\Bigg\{- \frac{b_j^4}{d^5} w\fint_{\mathbb{T}}\bigg(1- \frac{(\overline{\tau}-\overline{w})^2}{(\tau-w)^2}\bigg)\bigg(2\overline{\tau}^3+\frac{\overline{\tau}^2-\overline{w}^2}{\tau-w}\bigg)d\tau
\\&- \frac{b_j^4}{d^5}  w\fint_{\mathbb{T}}\bigg(w+\tau- \frac{(\overline{\tau}-\overline{w})(\overline{\tau}^2-\overline{w}^2)}{(\tau-w)^2}\bigg)\Big(\overline{\tau}^2+\frac{\overline{\tau}-\overline{w}}{\tau-w}\Big)d\tau \Bigg\}.
\end{align*}
Now using the fact that
\begin{equation*}
\overline{\tau}-\overline{w}=-(\tau-w)\overline{w}\overline{\tau}\quad\forall w,\tau\in \mathbb{T},
\end{equation*} 
we get
\begin{align}\label{parepsggg2g1counter}
&\partial^3_{\varepsilon g g} F_j(0,g_0)[g''(0),g'(0)]=-\frac{\gamma_1b_j^4}{d^5} \textnormal{Im}\Big\{{w}\Big\}.
\end{align}
%
%
From \eqref{epsepsgcounter} we find
%
\begin{align*}
\partial^3_{\varepsilon\varepsilon g g} F_j(0,g_0)[h,k]&=-\frac{2\gamma_1}{d}b_{3-j}^2\textnormal{Im}\Bigg\{w\fint_{\mathbb{T}}{   k'_{3-j}(\tau)\overline{ h_{3-j}(\tau)}}d\tau+w\fint_{\mathbb{T}}{   h'_{3-j}(\tau)\overline{ k_{3-j}(\tau)}}d\tau
 \Bigg\}.
\end{align*}
Therefore 
\begin{align}\label{epsepsggg1g1counter}
\partial^3_{\varepsilon\varepsilon g g} F_j(0,g_0)[g'(0),g'(0)]&=\frac{4\gamma_1b_{3-j}^4}{d^5}\textnormal{Im}\Big\{w \Big\}.
\end{align}
Plugging in the identities \eqref{epsepsggg1g1}, \eqref{parepsepsepsgcounter}, \eqref{parepsepsg2counter}, \eqref{parepsggg2g1counter} and \eqref{epsepsggg1g1counter}  into \eqref{varepsf41} we get
\begin{align}\label{varepsf41f}
g^{(4)}(0)=-\frac{24}{d^{4}}D_{g} F(0,g_0)^{-1}.
\Bigg(&-\begin{pmatrix}
b_1^4\\
 b_2^4 
\end{pmatrix}\textnormal{Im}\big\{ {w}^{5}\big\}+\frac{\gamma_1}{d}\begin{pmatrix}
 b_1^4+2b_2^4\\
 b_2^4 +2b_1^4
\end{pmatrix}\textnormal{Im}\big\{ {w}\big\}\notag\Bigg).
\end{align}
By virtue of \eqref{g-1counter} we conclude that
\begin{align}
g^{(4)}(0)&= \frac{24}{d^{4}}\bigg(\frac{\gamma_1}{2d}(2b_1+b_2^4),
\gamma_1(b_1^4-b_2^4),
-\frac{1}{4} b_1^4\overline{w}^4,
-\frac{1}{4} b_2^4\overline{w}^4
\bigg),
\end{align}
which achieves the proof of the desired result.
\end{proof}


\end{document}